\documentclass[a4paper,11pt]{amsart}

\usepackage[utf8]{inputenc}
\usepackage[english]{babel}
\usepackage{amsfonts}
\usepackage{amsmath}
\usepackage{mathrsfs}
\usepackage{amssymb}
\usepackage{amsthm}
\usepackage{amscd}
\usepackage{latexsym}
\usepackage{amstext}
\usepackage{amsxtra}
\usepackage[alphabetic,backrefs,lite]{amsrefs} 
\usepackage{color}
\usepackage{paralist}
\usepackage[colorinlistoftodos]{todonotes}
\usepackage[all]{xy}
\usepackage{wrapfig}

\usepackage{enumerate}
\usepackage{multirow}
\usepackage{wasysym}
\usepackage{latexsym} 
\usepackage{comment}
\usepackage{colonequals}

\usepackage[colorinlistoftodos]{todonotes}

\usepackage{geometry}
\usepackage{hyperref}
\hypersetup{
        citecolor=blue,
        colorlinks, 
        pdfauthor={D. Bricalli, F. F. Favale, G. P. Pirola},
        pdftitle={On the Hessian of cubic hypersurfaces},
        linktocpage
}

\newtheorem{theorem}{Theorem}[section]
\newtheorem*{theorem*}{Theorem}
\newtheorem{lemma}[theorem]{Lemma}
\newtheorem{corollary}[theorem]{Corollary}
\newtheorem{proposition}[theorem]{Proposition}
\newtheorem*{proposition*}{Proposition}
\newtheorem*{thmA*}{Theorem A}
\newtheorem*{thmB*}{Theorem B}
\newtheorem*{thmC*}{Theorem C}
\newtheorem*{thmD*}{Theorem D}
\newtheorem{remark}[theorem]{Remark}

\newtheorem{definition}[theorem]{Definition}
\newtheorem{example}[theorem]{Example} 
\newtheorem{conj}[theorem]{Conjecture}

\usepackage{xpatch}
\makeatletter
\AtBeginDocument{\xpatchcmd{\@thm}{\thm@headpunct{.}}{\thm@headpunct{}}{}{}}
\makeatother

\newcommand{\nc}{\newcommand} 
\nc{\cH}{{\mathcal H}}
\nc{\cA}{{\mathcal A}}
\nc{\cG}{{\mathcal G}}
\nc{\cC}{{\mathcal C}}
\nc{\cD}{{\mathcal D}}
\nc{\cO}{{\mathcal O}}
\nc{\cI}{{\mathcal I}}
\nc{\cB}{{\mathcal B}}
\nc{\cY}{{\mathcal Y}}
\nc{\cK}{{\mathcal K}} 
\nc{\cX}{{\mathcal X}}
\nc{\cS}{{\mathcal S}}
\nc{\cE}{{\mathcal E}}
\nc{\cF}{{\mathcal F}}
\nc{\cZ}{{\mathcal Z}}
\nc{\cQ}{{\mathcal Q}}
\nc{\cN}{{\mathcal N}}
\nc{\cP}{{\mathcal P}}
\nc{\cL}{{\mathcal L}}
\nc{\cM}{{\mathcal M}}
\nc{\cT}{{\mathcal T}}
\nc{\cW}{{\mathcal W}}
\nc{\cU}{{\mathcal U}}
\nc{\cJ}{{\mathcal J}}
\nc{\cV}{{\mathcal V}}
\nc{\bH}{{\mathbb H}}
\nc{\bA}{{\mathbb A}}
\nc{\bG}{{\mathbb G}}
\nc{\bC}{{\mathbb C}}
\nc{\bO}{{\mathbb O}}
\nc{\bI}{{\mathbb I}}
\nc{\bB}{{\mathbb B}}
\nc{\bY}{{\mathbb Y}}
\nc{\bK}{{\mathbb K}} 
\nc{\bX}{{\mathbb X}}
\nc{\bS}{{\mathbb S}}
\nc{\bE}{{\mathbb E}}
\nc{\bF}{{\mathbb F}}
\nc{\bZ}{{\mathbb Z}}
\nc{\bQ}{{\mathbb Q}}
\nc{\bN}{{\mathbb N}}
\nc{\bP}{{\mathbb P}}
\nc{\bL}{{\mathbb L}}
\nc{\bM}{{\mathbb M}}
\nc{\bT}{{\mathbb T}}
\nc{\bW}{{\mathbb W}}
\nc{\bU}{{\mathbb U}}
\nc{\bD}{{\mathbb D}}
\nc{\bJ}{{\mathbb J}}
\nc{\bV}{{\mathbb V}}
\nc{\bbZ}{{\mathbb Z}}
\nc{\bR}{{\mathbb R}}
\nc{\fr}{{\rightarrow}}
\nc{\co}{{\nabla}}

\nc{\cu}{{\barline{\nabla}}}

\usepackage{mathtools}

\DeclareMathOperator{\codim}{codim}

\DeclareMathOperator{\Ann}{Ann}

\DeclareMathOperator{\hess}{hess}
\DeclareMathOperator{\Rank}{Rank}
\DeclareMathOperator{\Sing}{Sing}

\DeclareMathOperator{\Edim}{Edim}

\DeclareMathOperator{\Ker}{Ker}
\DeclareMathOperator{\coker}{coker}
\DeclareMathOperator{\GL}{GL}
\DeclareMathOperator{\SL}{SL}
\DeclareMathOperator{\Sym}{Sym}
\DeclareMathOperator{\id}{id}
\DeclareMathOperator{\Hom}{Hom}
\DeclareMathOperator{\Pic}{Pic}

\newcommand{\pa}[1]{{\partial_{#1}}}

\title[On the Hessian of cubic hypersurfaces]{On the Hessian of cubic hypersurfaces}

\author{Davide Bricalli}
\address{Dipartimento di Matematica,
	Universit\`a degli Studi di Pavia,
	Via Ferrata, 5
	I-27100 Pavia, Italy}
\email{davide.bricalli@unipv.it}

\author{Filippo Francesco Favale}
\address{Dipartimento di Matematica,
	Universit\`a degli Studi di Pavia,
	Via Ferrata, 5
	I-27100 Pavia, Italy}
\email{filippo.favale@unipv.it}

\author{Gian Pietro Pirola}
\address{Dipartimento di Matematica,
	Universit\`a degli Studi di Pavia,
	Via Ferrata, 5
	I-27100 Pavia, Italy}
\email{gianpietro.pirola@unipv.it}

\date{\today}
\thanks{
\\
\noindent {\bf Acknowledgements}: \\
The first author wants to express his gratitude to Robert Lazarsfeld for helpful discussion on degeneracy loci and other topics. The authors are thankful to Letterio Gatto for some insights about the contents of the last part of the paper. The authors are partially supported by INdAM - GNSAGA and by PRIN \emph{``Moduli spaces and Lie theory''} and by (MIUR): Dipartimenti di Eccellenza Program (2018-2022) - Dept. of Math. Univ. of Pavia. The first and second author are partially supported by the INdAM – GNSAGA Project, “Classification Problems in Algebraic Geometry: Lefschetz Properties and Moduli Spaces” (CUP$\_$E55F22000270001)
}


\begin{document}

\maketitle

\begin{abstract}
In this paper, we analyze the Hessian locus associated to a general cubic hypersurface, by describing for every $n$ its singular locus and its desingularization. The strategy is based on strong connections between the Hessian and the quadrics defined as partial derivatives of the cubic polynomial. In particular, we focus our attention on the singularities of the Hessian hypersurface associated to the general cubic fourfold. It turns out to be a minimal surface of general type: its analysis is developed by exploiting the nature of this surface as a degeneracy locus of a symmetric vector bundle map and by describing an unramified double cover, which is constructed in a more general setting.
\end{abstract}

\section*{Introduction}

Consider an algebraically closed field $\bK$ of characteristic $0$. Given a non zero homogeneous polynomial $f\in\bK[x_0,\cdots,x_n]=S$ of degree $d$, we let $X=V(f)\subset \bP^n$ be the associated hypersurface and $H_f$ be its Hessian matrix. The determinant $h_f=\det(H_f)$ of $H_f$ is the hessian polynomial of $f$, which is either identically zero or a homogeneous polynomial of degree $(d-2)(n+1)$. The first case occurs for example if $V(f)$ is a cone, i.e. when the partial derivatives of $f$ are linearly dependent and, if $n\leq 3$, this is the only possibility, by the classical Theorem of Gordan and Noether. In the other case, $h_f$ defines a hypersurface $\cH_f=V(h_f)$ which is said to be the Hessian hypersurface associated to $X$. 
\smallskip

When dealing with smooth cubic hypersurfaces, the one we are interested in, one has a rich geometry which reflects also on Hessian varieties. For example, using Dolgachev's language (see \cite{Dol}), the {\em Hessian hypersurface equals the Steinerian hypersurface}
\begin{equation} 
\label{EQ:corr} 
\Gamma_{f}=\{([x],[y])\in\bP^n\times\bP^n \ | \ H_f(x)\cdot y=0\}.
\end{equation} 
One can show that $\Gamma_{f}$ is symmetric, in the sense that it is fixed by the natural involution $\tau([x],[y])=([y],[x])$, and that the first projection $\pi_1:\Gamma_f\rightarrow\cH_f$ gives a surjective morphism over $\cH_f$.

Secondly, the coeffiecients of the Hessian matrix $H_f$ are linear forms and $H_f$ is strictly related to quadrics in the Jacobian ideal $J_f$ of $f$. More precisely, by defining the loci 
$$\cD_k(f)=\{[x]\in\bP^n \ | \ \Rank(H_f(x))\le k\}\subseteq \cH\subseteq \bP^n,$$
one has that these can be identified with the intersections $\cQ_k\cap\bP(J^2_f)$, where $\cQ_k$ is the locus of quadrics in $\bP^n$ of rank at most $k$. 
\medskip

Hessians of cubic hypersurfaces have been studied by many authors, especially in the low dimensional case. The reader can refer to \cite{Dol} or \cite{GR}, for example. Just to mention some results and avoiding the huge amount of results on cubic curves, for $n=3$, it is known that the Hessian quartic surface associated to a general cubic surface is singular in exactly $10$ isolated points (see, for example, in \cite{DVG} and \cite{Hut}). \\
In a remarkable series of appendices (see \cite{AR}), Adler deeply studies the case $n=4$:
\begin{theorem*} If $X=V(f)\subset \bP^4$ is a general cubic threefold then  
\begin{enumerate}[(a)]
\item the correspondence \eqref{EQ:corr} gives a natural desingularization of the Hessian hypersurface;
\item $\cH_f$ is singular along a curve $C$;
\item $C$ is smooth and irreducible;
\item $C$ has degree $20$ and genus $26$.
\end{enumerate}
\end{theorem*}
The aim of this paper is to extend some of the above results in dimension greater than $3$.
Generalizing one of the results in \cite{AR}, we will prove the following

\begin{thmA*}[Singularities of the Hessian locus]
For any $[f]\in \bP(S^3)$ such that $X=V(f)$ is smooth, the Hessian variety is reduced and $\Sing(\cH_f)$ coincides with $\cD_{n-1}(f)$.
\end{thmA*}

We stress that the result holds for any smooth cubic and not only for the general one. Moreover, notice that the above conclusions can be false as soon as $X$ is not smooth (see Example \ref{REF:EXCUSP}).
\smallskip

The expected codimension of the loci $\cD_k(f)$ is $\binom{n+2-k}{2}$ so the singular locus of $\cH_f$ is not empty for any cubic form $f$ as soon as $n\geq3$. As for the case of threefolds, there is a natural desingularization for $f$ general: 

\begin{thmB*}[Correspondence and desigularization]
For the general smooth cubic hypersurface $V(f)$, $\Gamma_{f}$ is smooth and the natural projection $\pi_1:\Gamma_f\rightarrow\cH_f$ is a desingularization.
\end{thmB*}

The analysis of the Hessian loci is related to the theory of standard Artinian Gorenstein algebras (SAGAs, for brevity), for example in relation to the validity of some Lefschetz property (see, for example, \cite{HMNW}, \cite{MMN},\cite{MMO}, \cite{AR19}, \cite{DGI20}, \cite{DI22}  or \cite{bookLef} for a deep treatment). 
Furthermore, in this setting a correspondence as the above $\Gamma_f$ has been recently used in \cite{BFP} and \cite{BF} for giving a new proof of the above-mentioned theorem of Gordan and Noether and for proving the strong Lefschetz property for complete intersection SAGAs in codimension $5$ presented by quadrics and the strong Lefschetz property in degree $1$ for the same algebras with codimension $6$. \smallskip

Viceversa, the study of SAGAs can provide tools for analyzing Hessian loci. Indeed, by Macaulay theory of inverse systems, every SAGA can be written as $A_g=D/\Ann_D(g)=\oplus_{i=0}^{\deg(g)}A^i$ where $g\in S$ and $D=\bK[y_0,\cdots,y_n]$, with $y_i=\frac{\partial}{\partial x_i}$, is the ring of linear differential operators. When $V(f)$ is a smooth cubic, the Hessian hypersurface can be identified with the non-Lefschetz locus of $A_f$, i.e. the subvariety of $\bP(A^1)\simeq\bP^n$ of elements $v$ whose multiplication map $v\cdot:A^1\rightarrow A^2$ has non-trivial kernel. The interested reader can refer to \cite{BMMN}, \cite{AR19}, and \cite{BF}, where this locus has been studied.
\medskip

With the natural identification between $\bP^n$ and $\bP(A^1)$, one can give a very transparent description of the loci $\cD_k(f)$, namely
$$\cD_k(f)=\{[y]\in \bP(A^1)\,|\, \Rank(y\cdot:A^1\to A^2)\leq k\}.$$
With this description, we provide a very clean proof of the following:

\begin{thmC*}[Singularities for general cubic]
For a general smooth cubic hypersurface $V(f)$, if $\cD_k(f)\setminus\cD_{k-1}(f)$ is non-empty, then $\Sing(\cD_k(f))=\cD_{k-1}(f)$.
\end{thmC*}

Moreover, there is another description of these loci in the context of degeneracy loci of symmetric maps of vector bundles. We recall that given a vector bundle $E$ and a line bundle $L$ over a projective variety $X$, a  vector bundle map $\varphi:E\rightarrow E^*\otimes L$ is symmetric if $\varphi^\ast=\varphi.$ The degeneracy loci associated to $\varphi$ are defined as
$$\cD_k'(\varphi)=\{x\in X \ | \ \Rank(\varphi_x)\le k\}$$ 
and they have been studied in several works (for example in \cite{FL}, \cite{FL_Positive}, \cite{HT_CH}, \cite{HT_SS}, \cite{HT_CONN_OddRanks}, \cite{Laz}, \cite{Tu85}, \cite{Tu_CONN}, \cite{Tu_Even}).  By considering the vector bundle morphism
$$\varphi:\cO_{\bP^n}^{n+1}\to \cO_{\bP^n}^{n+1}(1)$$ 
defined by $H_f$ with $f\in S^3$ one easily sees that $\cD_k(f)=\cD_k'(\varphi)$.

With this description we can then compute the relevant Chern classes and show the non-emptiness and the connectedness properties of $\cD_k(f)$. 

The simplest case, not still analyzed in the literature, is the one of smooth cubic fourfolds in $\bP^5$. From the formula for the expected dimension, one can observe that it is also the last case where the singular locus can be smooth: as soon as $n\geq6$ the singular locus of the Hessian hypersurface of a cubic in $\bP^n$ is itself singular. One has the following:  

\begin{thmD*}[Hessian of cubic fourfolds: geometric invariant of the singular locus]
Let $X=V(f)$ be a general smooth cubic fourfold defined over $\bC$. Then, 
\begin{enumerate}[(a)]
\item the singular locus $Y=\Sing(\cH_f)=\cD_4(f)\subset \bP^5$ is a smooth surface;
\item $Y$ is an irreducible and minimal surface of general type with degree $35$ and canonical divisor $K_Y=3H|_Y+\eta$, where $H$ is the hyperplane class in $\bP^5$ and $\eta$ is a non-trivial $2$-torsion element in $\Pic^0(Y)$. 
\item $Y$ has topological Euler characteristic $e(Y)=357$, irregularity $q=0$ and 
$\chi(Y)=56$.
\end{enumerate}
\end{thmD*}

We would like to stress some remarks about the above statement:
\begin{itemize}
\item We assume $\bK=\bC,$ since we use singular cohomology, anyway the general case follows from Lefschetz' principle.
\item Theorem D still holds for $Y=\cD_4'(\varphi)$ where $\varphi:\cO_{\bP^5}^{6}\to \cO_{\bP^5}^{6}(1)$ is a general symmetric vector bundle map not necessarily induced by a hessian matrix.
\end{itemize}

More importantly, we would like to point out the most intriguing part of our work. The computation of the canonical divisor and the irregularity of $Y$ does not follow from {\cite {HT_CH}}. To detect the $2$-torsion of the cohomology we have to use a geometric construction valid for even $k$. Indeed, when $k=2m$ there is a natural unramified double covering $\tilde{Y}$ of $Y=\cD_{2m}'(\varphi)\setminus \cD_{2m-1}'(\varphi)$. For any $x\in Y$ there are two families of projective spaces of maximal dimension contained in the quadric defined by the symmetric map $\varphi_x$. The connected components of these families are two points associated to $x$ and this defines $\tilde{Y}$. 
Finally one has to show that this covering is connected. In order to do this, we study the complete family of projective spaces, that turns out to be a zero section of a non ample vector bundle, using Bott's vanishing theorem and some basic representation theory.
\bigskip




The plan of the article is the following. In Section \ref{SEC:Notations}, we will set the notation and derive the first basic results, based on the nature of the Hessian matrix of a general cubic form $f\in S=\bK[x_0,\cdots,x_n]$. We will introduce the loci $\cD_k(f)$ and compute their expected dimension. In Section \ref{SEC:Singanddesing}, we will introduce the correspondence $\Gamma_f$ and we will prove Theorem A and Theorem B. Section \ref{SEC:stratsmooth} will be devoted to the proof of Theorem C. In Section \ref{SEC:Degeneracyloci}, we will deal with degeneracy loci of symmetric morphisms of vector bundles. In this section, we will also give the general construction for the double cover mentioned above. Moreover, we prove that the double cover of $\cD_4(\varphi)'\subseteq \bP^5$ is non-trivial. In the last Section \ref{SEC:SingHessCubicFourfold}, we will describe the geometry of the surface $Y$ arising as the singular locus of the Hessian hypersurface $\cH_f$ of a general cubic fourfold, by showing Theorem D. We conclude the section by giving some additional information about the geometry of $Y$ by exploiting a computer algebra software.


\section{Hessians and quadrics}
\label{SEC:Notations}
Let $\bK$ be an algebraically closed field of characteristic $0$ and consider the projective space $\bP^n$ with $n\geq 2$. The homogeneous coordinate ring of $\bP^n$ is $S=\bK[x_0,\dots, x_n]=\bigoplus_{k\geq0}S^k$ where $S^k=H^0(\cO_{\bP^n}(k))$. 
Consider the graded algebra $D$ of linear differential operators on $S$, namely, $$D=\bK[y_0,\dots, y_n]=\bigoplus_{k\geq 0}D^k$$ where we set $y_i:=\frac{\pa{\,}}{\partial x_i}$.
This comes with the natural pairing $S\times D\to S$ which induces an isomorphism $(S^k)^*\simeq D^k$ for all $k$. This allows us to identify $\bP^n$ with $\bP((S^1)^*)$ and $\bP(D^1)$. From this, through the article, if $v=(v_0,\dots, v_n)\in D^1$ we will write $v(g)$ in order to mean $\sum_{k=0}^{n}v_ky_k(g)\in S$
for any $g\in S$. Moreover, where no confusion arises, if $v,w\in D^1$ we will simply write $vw(g)$ in order to mean $v(w(g))$.
\medskip

For each $f\in S^d$ one can define the gradient $\nabla(f)=(y_i(f))_{i=0,\dots,n}\in (S^{d-1})^{\oplus n+1}$ and the Hessian matrix of $f$ and the hessian of $f$, i.e 
$$H_f=\big((y_iy_j)(f)\big)_{i,j=0,\dots,n}\in \Sym^2\big((S^{d-1})^{\oplus n+1}\big)\qquad\mbox{ and }\qquad h_f=\det(H_f)\in S^{(n+1)(d-2)}.$$
Moreover, let us define the \emph{Hessian hypersurface} associated to $f$ as $\cH_f:=V(h_f).$
For any $d\geq 2$ we consider the subloci of $\bP(S^d)$ given by
$$\cC_{sing}=\{[f]\,|\, V(f) \mbox{ is singular}\}\qquad \cC_{cone}=\{[f]\,|\, V(f) \mbox{ is a cone}\}\quad \mbox{ and }\quad\cC_{GN}=\{[f]\,|\, h_f=0\}.$$
The latter is called the Gordan-Noether locus and it is well known that 
$$\cC_{cone}\subseteq \cC_{GN}\subseteq \cC_{sing}$$ 
and that $\cC_{sing}$ is a divisor in $\bP(S^d)$. Moreover, the first inclusion is strict unless $d=2$ or, by the Gordan-Noether's theorem, $d\geq 3$ and $n\leq 3$. The second inclusion is an equality for $d=2$, but is again strict for $d\geq 3$.
\medskip

Before moving on, let us recall the differential Euler identity (see \cite[Lemma 7.2.19]{Rus}), namely the relation
\begin{equation}
\label{EQ:DiffEuler}
v^m(G)=\left(\sum_{k=0}^n v_ky_k\right)^m(G)=m!\cdot G\left(v_0,\dots,v_n\right)
\end{equation}
where $v=(v_0,\dots, v_n)\in D^1$ and $G\in S^m$. 
\medskip

We summarize here some easy and probably well known results concerning the case $d=3$. Some of them are treated, for example, in \cite{Dol}.
\begin{lemma}
\label{LEM:MAGICHESS}
Let $f$ be an element in $S^3$.  
\begin{enumerate}[(a)]
\item For all $v,w\in \bK^{n+1}$ we have $H_f(v)\cdot w=\nabla(vw(f))$. In particular, $H_f(v)\cdot w=H_f(w)\cdot v$.
\item For all $v\in \bK^{n+1}$ one has $2\nabla(f)(v)=H_f(v)\cdot v$. In particular, assuming $f\neq 0$, $[v]\in \bP^n$ is singular for $V(f)$ if and only if $H_f(v)\cdot v=0$ or, equivalently, $v^2(f)=0$.
\item For all $v,w\in \bK^{n+1}$ we have $w^T\cdot H_f(v)\cdot w=2(v(f))(w)$.
\end{enumerate}
\end{lemma}

\begin{proof}
\textbf{(a)}$\quad$ Since $f\in S^3$, we have $vw(f)\in S^1$. Notice that an element $g\in S^1$ is identified by its gradient $\nabla(g)$ by the differential Euler relation. More precisely, one can easily see that if $g=\sum_{k}a_k x_k$ then $g(e_k)=a_k=y_k(g)$. From this, by $\bK$-bilinearity, it is enough to consider the case $v=e_i$ and $w=e_j$. We have
$$H_f(e_i)\cdot e_j=(H_f(e_i))^{j}=((H_f)_{kj}(e_i))_{k=0}^{n}= (y_iy_jy_k(f))_{k=0}^{n}=\nabla(y_iy_j(f))$$
where, if $M$ is a matrix, we set $M^{j}$ to be its $j$-th column and $M_{ij}$ to be the $i$-th entry of $M^j$.

\textbf{(b)}$\quad$ From (a) we have 
$$H_f(v)\cdot v=\nabla(v^2(f))=(y_kv^2(f))_{k=0}^n=(v^2y_k(f))_{k=0}^n.$$
Now, since $y_k(f)\in S^2$, by the Euler differential identity we have $v^2(y_k(f))=2y_k(f)
(v)$, which proves the claim. 

\textbf{(c)}$\quad$ Using (a), (b) and the symmetry of $H_f$, we obtain
$$w^T\cdot H_f(v)\cdot w=w^T\cdot H_f(w)\cdot v=v^T\cdot H_f(w)\cdot w=2v^T\cdot \nabla(f)(w).$$
On the other hand, $v^T\cdot\nabla(f)=v(f)$, so we get the claim.
\end{proof}

In particular, from $(c)$ a cubic $V(f)\subseteq \bP^n$ is a cone if and only if there exists $[v]\in \bP^n$ such that $H_f(v)\equiv 0$. \smallskip

For any $f\in S^d$, one can define the {\it {Jacobian ideal of $f$}}, i.e. the graded ideal $J_f=\bigoplus_{k\geq d-1}J_f^k$ spanned by the partial derivatives of $f$. One can also associate to $f$ its {\it {apolar ring}} $A_f=D/\Ann_D(f)$ (where $\Ann_D(f)$ is the annihilator of $f$ in $D$), which is a graded Artinian algebra with socle in degree $d$. We will denote by $A_f^k$, or simply $A^k$ if no confusion arises, the homogeneous part of degree $k$ of $A_f$, i.e. $D^k/(D^k\cap \Ann_D(f))$.

Moreover, one can show that $A_f$ is standard (i.e. it is generated in degree $1$) and it satisfies Gorenstein duality (namely, $A^d$ has dimension $1$ and the multiplication map $A^{k}\times A^{d-k}\to A^d\simeq \bK$ is a perfect pairing). Notice that, by the characterisation of cones (see, for example, \cite[Prop. 7.1.2]{Rus}), $V(f)$ is not a cone if and only if $\Ann_D(f)\cap D^1=\{0\}$, i.e. if and only if $A^1=D^1$. In this case, then, we have a natural identification between $\bP(D^1)$ and $\bP(A^1)$.
\medskip

From now on we will only deal with the case $d=3$. We consider the loci
\begin{equation}
\label{EQ:DefDk}
\cD_k(f)=\{[x]\in \bP^n\, |\, \Rank(H_f(x))\leq k\}=\{[y]\in \bP(A^1)\, |\, \Rank(y\cdot: A^1\to A^2)\leq k\}
\end{equation}
where the second equality easily follows from Lemma \ref{LEM:MAGICHESS}(a) under the above mentioned identification between $\bP^n$ and $\bP(A^1)\simeq \bP(D^1)$.
We are interested in studying these loci for a general $[f]\in \bP(S^3)$. 
\medskip

If $X=V(f)$ is a cubic hypersurface, which is not a cone, the associated apolar ring $A_f$ codifies all the information about the cubic $X$ itself and its Hessian $\cH_f$.

\begin{proposition}
\label{PROP:INUTIL}
Given a cubic hypersurface $X=V(f)$ (not a cone) and the corresponding $A_f$, we have 
\begin{enumerate}[(a)]
    \item $X=\{[y]\in \bP(A^1)\,|\, y^3=0\}$;
    \item $\Sing(X)=\{[y]\in \bP(A^1)\,|\, y^2=0\}$;
    \item $\cH_f=\cD_n(f)=\{[y]\in \bP(A^1)\,|\,\, \exists\, [x]\in \bP(A^1) \mbox{ with } xy=0\}$.
\end{enumerate}
\end{proposition}

\begin{proof}
The claims follow directly from Lemma \ref{LEM:MAGICHESS} by using the differential Euler identity. By the definition of the loci $\cD_k(f)$, it is clear that $\cH_f=\cD_n(f)$.  
\end{proof}

Notice that, by Proposition \ref{PROP:INUTIL}, the Hessian locus corresponds to the so called {\it{non-Lefschetz locus}} of $A_f$, i.e. the set of points for which the validity of the weak Lefschetz property fails for $A_f$ (for a treatment of this topic the interest reader can refer to \cite{BMMN}, \cite{AR19} and \cite{BF}).
\medskip

The loci $\cD_k(f)$ give a natural stratification not only of the whole projective space, but also of the Hessian locus $\cH_f$. As a consequence of Lemma \ref{LEM:MAGICHESS} we have that for any $[f]\not \in \cC_{cone}$ we have a commutative diagram
\begin{equation}
\label{EQ:LinearEmbeddings}
\xymatrix{
& \bP^n \ar@/_1pc/@{^{(}->}[dl]_-{\tau_1} \ar@/^1pc/@{_{(}->}[rd]^-{\tau_2} & \\
\bP(S^2) & & \bP(\Sym^2(\bK^{n+1})) \ar[ll]_-{\simeq}
}
\qquad 
\xymatrix{
& [v] \ar@/_1pc/@{|->}[dl] \ar@/^1pc/@{|->}[rd] & \\
[v(f)] & & [H_f(v)] \ar@{|->}[ll]_{\simeq}
}
\end{equation}
where the diagonal arrows are linear embeddings of $\bP^{n}$, while the horizontal map is the canonical isomorphism $[M]\mapsto [x^T M x]$ which identifies a symmetric matrix $M$ with the quadratic form represented by $M$.
Since the rank of a symmetric matrix $M$ with coefficients in $\bK$ and of the quadratic form $x^T M x$ are the same, it can be useful to study the image of $\cD_k(f)$ via the linear embedding $\tau_1$ (or $\tau_2$). Passing from one map to the other will be useful to catch different features of the objects we want to study. Let us observe that the image of $\tau_1$ is exactly $\bP(J_f^2)$: if we define
$$\cQ_k=\{[q]\in \bP(S^2)\, |\, \Rank(q)\leq k\}$$
it is clear that $\tau_1(\cD_k(f))=\bP(J_f^2)\cap \cQ_k$.

In what follows, for brevity, we will not specify the linear embedding $\tau_i$ (for $i=1,2$) in the identification of the loci $\cD_k$ with their images. Moreover, we will write simply $\cD_{k}$ instead of $\cD_k(f)$ 
when it is clear from the context which $[f]\in \bP(S^3)$ we are considering.
In light of this we recall some important facts about the geometry of $\cQ_k$.

\begin{lemma}
\label{LEM:proprQk}
For any $1\leq k\leq n+1$, $\cQ_k$ is a closed subvariety of $\bP(S^2)$. Moreover
\begin{itemize}
    \item We have $\codim_{\bP(S^2)}\cQ_k=\binom{n+2-k}{2}$ and $\dim\cQ_k=kn-\frac{(k-1)(k-2)}{2}$;
    \item The degree of $\cQ_k$ as variety inside $\bP(S^2)$ is given by the formula
    $$\deg(\cQ_k)=\prod_{t=0}^{n-k}\frac{\binom{n+t+1}{n-k-t+1}}{\binom{2t+1}{t}}.$$
    \item For $1\leq k\leq n$, the singular locus of $\cQ_k$ coincides with $\cQ_{k-1}$.
\end{itemize}
\end{lemma}
\begin{proof}
See \cite[Chapter 22]{Har_book} and \cite{HT_SS} for the formula of the degree of $\cQ_k$.
\end{proof}

Notice that from the above description, it is clear that $\cD_{k-1}\subseteq \Sing(\cD_k)$ and one might expect that the equality holds. Actually, we will prove that this is true when $f$ is general (see Theorem \ref{THM:smooth2}and Corollary \ref{COR:DIMSING}) although it does not hold for all $[f]\not\in \cC_{sing}$ (see Remark \ref{REM:FailThmsmooth2}).

\begin{remark}
If we consider an element $[f]\in\cC_{GN}$, since in this case $h_f\equiv0$, we have that $\cD_n(f)$ is the whole projective space $\bP^n$, i.e. $\bP(J^2)\subset\cQ_n$. 
\end{remark}

Let us then take $[f]\in\bP(S^3)$ such that $V(f)$ is smooth and consider the first level of the stratification of $\cH$ presented above, i.e. the variety $\cD_{n-1}(f)$. We can immediately observe that in $\bP(S^2)$, which has dimension $\binom{n+2}{2}-1$, we have the subspaces $\bP(J_f^2)$ and $\cQ_{n-1}$, whose dimensions are $n$ and $\binom{n+2}{2}-4$ respectively (from Lemma \ref{LEM:proprQk}). Hence, we easily get that the expected dimension of $\cD_{n-1}(f)=\bP(J^2)\cap\cQ_{n-1}$ is
$$\Edim(\cD_{n-1})=\binom{n+2}{2}-4+n-\left(\binom{n+2}{2}-1\right)=n-3.$$

Since $\cD_{n-1}\subseteq\Sing(\cD_n)=\Sing(\cH)$ and $\dim(\cD_{n-1})\geq n-3$, we have the following:

\begin{proposition}\label{PROP:Hesssing}
For all $[f]\not\in\cC_{sing}$, the Hessian hypersurface $\cH_f$ has singular locus of dimension at least $n-3$ (i.e. $\Sing(\cH_f)$ has codimension at most $2$ in $\cH_f$). In particular, if $n\geq 3$, $\cH_f$ is always singular.
\end{proposition}

The inclusion $\cD_{n-1}\subseteq\Sing(\cH)$ can also be obtained by using Jacobi's formula, which controls the derivatives of the determinant of the Hessian matrix.
We will show, generalizing a result in \cite{AR}, that for all $[f]\not\in\cC_{sing}$ we actually have $\cD_{n-1}=\Sing(\cH)$ (see Theorem \ref{THM:SingH}) and that when $[f]$ is general, then $\cD_{n-1}$ has the expected dimension (see Section \ref{SEC:stratsmooth}). 

\begin{remark}
\label{RMK:example}
When $n\leq 4$, the above mentioned results are known. More precisely:
\begin{itemize}
    \item For $n=2$ it is well known that the Hessian curve associated to the general cubic plane curve is smooth.
    \item For $n=3$, the Hessian surface associated to the general cubic surface is singular in $10$ points, which are nodes for $\cH$ (see \cite{DVG});
    \item For $n=4$, Adler in \cite[Appendix IV]{AR} has shown that the Hessian hypersurface associated to the general cubic threefold is singular along a curve.
\end{itemize}
\end{remark}

We conclude this section with a remark about the expected dimension of $D_k(f)$.

\begin{remark}
\label{REM:EXPDIM}
As we have done above, one can argue that for $[f]\in \bP(S^3)$, 
$$\Edim(\cD_{k}(f))=n-\binom{n-k+2}{2}$$
is the expected dimension of $\cD_{k}(f)=\bP(J_f^2)\cap\cQ_{k}$. In particular the expected codimension of $\cD_k(f)$ is exactly the codimension of $\cQ_k$ in $\bP(S^2)$.
\end{remark}


\section{Singular loci and desingularizations}
\label{SEC:Singanddesing}

Set $U=\bP(S^3)\setminus\cC_{sing}$, i.e. the open set parametrizing smooth cubics in $\bP^n$ and consider
$[f]\in U$. As we have seen, the Hessian variety $\cH_f$ is singular with $\dim(\Sing(\cH_f))\geq n-3$ since it contains the locus $\cD_{n-1}(f)$. The aim of this section is twofold. Firstly, we want to show that {\it for all} $[f]\in U$ the singular locus of $\cH_f$ coincides with $\cD_{n-1}(f)$. Secondly, we want to describe a way to desingularize $\cH_f$ for $[f]\in U$ general. For both results, it will be central the following construction.
\smallskip

For any $[f]\in \bP(S^3)$ define $\Gamma_f$ (or $\Gamma$, if $f$ is clear from the context) as
\begin{equation}
\label{EQ:DefGamma}
\Gamma_f=\{([x],[y])\in \bP^n\times \bP^n \, | \, H_f(x)\cdot y=0\}=\{([x],[y])\in \bP(A^1)\times \bP(A^1) \, | \, xy=0\}.    
\end{equation}
We denote by $\pi_1$ and $\pi_2$ the natural projections from $\Gamma_f$ on the factors. Let us define, for any $[x]\in \bP^n$,
\begin{equation}
\iota([x])=\bP(\Ker(H_f(x))).
\end{equation}

Since the fiber of $\pi_1$ over $[x]\in \bP^n$ is $[x]\times \bP(\Ker(H_f(x)))$ we have a natural identification between such a fiber and $\iota([x])$. 
If we work with $\bP(A^1)$, we have $\iota([y])=\bP(\Ker(y\cdot:A^1\rightarrow A^2))$ for all $[y]\in \bP(A^1)$. Notice that $\iota([x])\neq \emptyset$ if and only if $[x]\in \cH_f$ and that
$\iota([x])\subset \cH_f$ since, for any $[y]\in \iota([x])$, we have $[x]\in\iota[y]$ by Lemma \ref{LEM:MAGICHESS}(a).

\begin{lemma}
\label{LEM:BasicFactsGamma}
The morphism $\tau([x],[y])=([y],[x])$ induces a natural involution on $\Gamma_f$ which acts freely on $\Gamma_f$ if and only if $[f] \not\in\cC_{sing}$. Moreover, the image of $\pi_i$ is the Hessian locus $\cH_f=\cD_n(f)$ and $\pi_i$ is an isomorphism over the open $\cH\setminus \cD_{n-1}(f)$.
\end{lemma}

\begin{proof}
The involution $\tau$ on $\bP^n\times \bP^n$ descends to an involution on $\Gamma$ since $H_f(v)\cdot w=H_f(w)\cdot v$ as proved in Lemma \ref{LEM:MAGICHESS}(a). A point $([v],[w])\in \Gamma$ is a fixed point if and only if $[v]=[w]$ so $\tau$ has a fixed point if and only if there exists $[v]\in \bP^n$ such that $H_f(v)\cdot v=0$. By Lemma \ref{LEM:MAGICHESS}(b), this happens if and only if $V(f)$ is singular.

By definition, $\pi_1^{-1}([v])$ is not empty if and only if we can find a non trivial element in $\Ker(H_f(v))$, i.e. if and only if the rank of $H_f(v)$ is not maximal. This happens exactly when $[v]\in \cH$ by definition of $\cH$. On $\cH^s=\cH\setminus \cD_{n-1}$ we have only points such that $\Rank(H_f(v))=n$ so $\Ker(H_f(v))$ has dimension $1$. Hence, $\pi_1|_{\pi_1^{-1}(\cH^s)}:\pi_1^{-1}(\cH^s)\to \cH^s$ is an isomorphism. The claim for the second projection follows since $\pi_{i}\circ \tau=\pi_{3-i}$, for $i=1,2$.
\end{proof}

The variety $\Gamma$ has been used in \cite{AR} in order to desingularize the Hessian locus for $n=4$. 
The approach used was to study a specific case, namely the case of the Klein cubic $f_0=x_0x_4^2+x_1x_0^2+x_2x_1^2+x_3x_2^2+x_4x_3^2$, and prove that $\Gamma$ is smooth. 
Then, the result holds also for $[f]\in U$ general. Unfortunately, this approach cannot be carried out completely for any $n$. Nevertheless, the methods used in \cite{AR} can be used and generalised in order to prove that $\Sing(\cH)=\cD_{n-1}(f)$ as we will do in the next subsection. Instead, we propose a different approach in order to describe the desingularization of $\cH$ for any $n$.


\subsection{Description of the singular locus}
In this subsection, we prove Theorem A. 




\begin{theorem}
    \label{THM:SingH}
For any $[f]\in U$ we have that 
$$\Sing(\cH_f)=\cD_{n-1}(f)$$
and that $\cH_f$ is reduced.
\end{theorem}

\begin{proof}
First of all, let us define $\cH^s=\cH\setminus \cD_{n-1}$, an open set of $\cH$ that can be described as
$$\cH^s=\{[x]\in\bP^n \ | \ \Rank(H_f(x))=n\}.$$
Observe that $\cH^s$ is not empty as, otherwise, the fiber of $\pi_1:\Gamma\to \cH$ over any point $[x]$ would have positive dimension. Hence, we would have a component of $\Gamma$ of dimension at least $n$. Since all varieties of $\bP^n\times \bP^n$ of dimension $n$ cut non-trivially the diagonal $\Delta$. This is impossible by Lemma \ref{LEM:MAGICHESS}(b) since $f\in U$. 
\medskip

As recalled before, it is enough to show that the inclusion $\Sing(\cH)\subseteq \cD_{n-1}$ holds, i.e. we show that each point in $\cH^s$ is a smooth point for $\cH$.
\smallskip

By setting $\Gamma^s=\pi_1^{-1}(\cH^s)$, using Lemma \ref{LEM:BasicFactsGamma} we have that $\pi_1|_{\Gamma^s}:\Gamma^s\rightarrow \cH^s$ is an isomorphism with inverse is $[x]\mapsto ([x],\bP(\Ker(H_f(x))))$. Then, we have that a point $[x]\in \cH^s$ is smooth in $\cH^s$ (so in $\cH$) if and only if the point $\pi_1^{-1}([x])=([x],[y])=([x],\bP(\mathrm{Ker}(H_f(x))))$ is smooth for $\Gamma$.
We can now consider the bihomogeneous lifting of $\Gamma$ to $\bK^{n+1}\times\bK^{n+1}$, i.e.
$$\tilde{\Gamma}=\{(x,y)\in \bK^{n+1}\times\bK^{n+1} \ | \ H_f(x)\cdot y=0\}.$$
To get the claim, we can then prove that $\tilde{\Gamma}$ is smoooth at $\tilde{p}$.
We can easily describe $\tilde{\Gamma}$ as the zero locus of a suitable function: indeed, we can write $\tilde{\Gamma}=V(\tilde{F})$, where 
$$\tilde{F}:\bK^{n+1}\times\bK^{n+1}\rightarrow \bK^{n+1} \qquad \tilde{F}((x,y))=H_f(x)\cdot y.$$
Let us now observe that, by Lemma \ref{LEM:MAGICHESS}$(a)$, we have the equality $H_f(x)\cdot y=H_f(y)\cdot x$; from this one easily gets that the Jacobian matrix of the map $\tilde{F}$ in $\tilde{p}=(x,y)\in\tilde{\Gamma}$ can be described as the block matrix
$$J(\tilde{F})(\tilde{p})=(H_f(y)\,|\,H_f(x)).$$
Let us now assume, by contradiction, that $\tilde{p}\in \tilde{\Gamma}$ is a singular point: this implies that the matrix $J(\tilde{F})(\tilde{p})$ is not of maximal rank (i.e. it has rank smaller or equal than $n$).
Since $\tilde{p}=(x,y)$ and $[x]\in \cH^s$, we have that $\Rank(H_f(x))=n$, and so $\Rank(J(\tilde{F})(\tilde{p}))=n$. 
Observe that if $A$ and $B$ are symmetric matrices with coefficients in a field, the block matrix $M=(A|B)$ has non maximal rank if and only if $\Ker(A)\cap\Ker(B)\neq\{0\}$. Then, $\tilde{p}$ is singular if and only if the intersection $\Ker(H_f(y))\cap\Ker(H_f(x))$ is not trivial and, in particular, equal to $\left\langle y\right\rangle$, since $\Ker(H_f(x))=\left\langle y\right\rangle$. Then we get that $H_f(y)\cdot y=0$, that is a contradiction by Lemma \ref{LEM:MAGICHESS}$(b)$, since we are considering $[f]$ smooth. Hence, $\tilde{p}$ is a smooth point of $\tilde{\Gamma}$, as claimed.

Finally, for the last claim, observe that if $\cH_f$ is non reduced, then $\cH_f=\Sing(\cH_f)$ and thus $\cH^s=\emptyset$. This is impossible as shown above.
\end{proof}

Let us stress that the hypotheses of smoothness for $f$ can't be ruled out, as one can see in the following:
\begin{example}
\label{REF:EXCUSP}
Consider the cuspidal cubic curve $X=V(f)\subset\bP^2$ defined by $f=x_0^2x_2-x_1^3$ Then $P_2=(0:0:1)$ is the cusp and $P_0=(1:0:0)$ is the only flex of $X$. In this case, $\hess(f)$ is proportional to $x_0^2x_1$ so $\cH_f$ splits as a double line $\ell=V(x_0)$, which is the cuspidal tangent, and $m=V(x_1)$, joining $P_0$ and $P_2$. One also easily checks that $\cD_1(f)=\{P_2,(0:1:0)\}$. Hence, in this case, $\cH_f$ is non-reduced and $\cD_1(f)\subsetneq \Sing(\cH_f)$.
\end{example}

The techniques used in \cite{AR} can be further generalized in order to characterize the cubics $[f]\in U$ such that $\Gamma$ is singular. We report the method here for completion, generalising it for any value of $n$.

\begin{definition}
Let $f\in \bP(S^3)$. We say that $T=([x],[y],[z])\in\cH_f^3$ is a triangle (in $\cH_f$) if $[x]\in \iota([y]), [y]\in \iota([z])$ and $[z]\in \iota([x])$.
\end{definition}
Notice that if $T$ is a triangle in $\cH_f$, then any permutation of the triple $T$ is so by Lemma \ref{LEM:MAGICHESS}(a) and all "vertices" of $T$ actually belong to $\Sing(\cH_f)$. Moreover, two "vertices" of $T$ cannot be equal unless $V(f)$ is singular by Lemma \ref{LEM:MAGICHESS}(b). We stress that if $T=([x],[y],[z])$ is a triangle, it is true that $\bP(\langle [x],[y]\rangle)\subseteq \iota([z])=\bP^s$ but it can happen that $s\geq 2$.

\begin{lemma}
For any $[f]\in \bP(S^3)$, we have that the point $([x],[y])$ is singular for $\Gamma_f$ if and only if there exists $[z]$ such that $T=([x],[y],[z])$ is a triangle in $\cH_f$.
\end{lemma}
\begin{proof}
Let us consider again the bihomogeneous lifting $\tilde{\Gamma}$ and let us assume that $(x,y)\in\tilde{\Gamma}$ with $x,y\neq 0$ is singular for $\tilde{\Gamma}$. As observed in the proof of the previous Theorem \ref{THM:SingH}, this is equivalent to asking that $\Rank(H_f(y)\ | \ H_f(x))<n$ and this happens exactly if there exists $[z]\in \iota([x])\cap \iota([y])$.
By Lemma \ref{LEM:MAGICHESS}(b), we also have that $[x]\in \iota([z])$ and $[y]\in \iota([z])$. But since $(x,y)$ is a point of $\tilde{\Gamma}$ we also have that $[x]\in \iota([y])$. Then, $T=([x],[y],[z])$ is a triangle in $\cH_f$.
\end{proof}


\subsection{Desingularizing the general Hessian hypersurface}

Adler in \cite{AR} uses this characterization of the singularities in $\Gamma$ through the existence of some triangle $T$ in $\cH$ to show the smoothness for the variety $\Gamma$ associated to the Klein cubic. This has be done by studying the rich and particular geometry of this specific cubic and by using also some refined geometrical and algebraic techniques. However, we think that this approach can not be easily exploited and generalized to any dimension. For this reason, in the this subsection we will present a different strategy.

\begin{theorem}
\label{THM:ResolutionOfHessian}
Let $[f]\in U$ be general. Then $\Gamma_f$ is smooth and $\pi_i:\Gamma_f\to \cH_f$ is a desingularization. 
\end{theorem}

\begin{proof}
Consider 
$$W=\{([v],[w],[f])\in \bP(D^1)\times \bP(D^1)\times U\quad |\quad vw(f)=0\}.$$
First of all, we want to show that $W$ is smooth. In order to do this, it is enough to consider the lift 
$$\tilde{W}=\{(v,w,f)\in D^1\times D^1\times \tilde{U}\quad |\quad v, w\neq 0,\,\, vw(f)=0\}$$
where $\tilde{U}=\{f\in S^3\setminus\{0\}\, | \, V(f) \mbox{ is smooth}\}$, and to prove that $\tilde{W}$ is smooth. To show that $\tilde{W}$ is smooth we will prove that the Zariski tangent space of $\tilde{W}$ has constant dimension in any point of $\tilde{W}$. 
\medskip

Fix $p=(v,w,f)\in \tilde{W}$. The Zariski tangent space of $\tilde{W}$ in $p$ is contained in $T_p(D^1\times D^1\times \tilde{U})=D^1\times D^1\times S^3$ and is described as the set of triples $(v',w',f')$ such that 
$$(v+tv')(w+tw')(f+tf') \equiv 0 \mod t^2.$$
Expanding this relation we have
$$(v+tv')(w+tw')(f+tf')=vw(f)+t\left(vw'(f)+v'w(f)+vw(f')\right)+t^2(\cdots)$$
so, for any $(v',w',f')\in D^1\times D^1\times S^3$ we have
\begin{equation}
\label{EQ:TanZarW}
(v',w',f')\in T_p\tilde{W} \qquad\Longleftrightarrow\qquad  (\star):\quad vw'(f)+v'w(f)+vw(f')=0.  
\end{equation}
Let $\sigma:\tilde{W}\to D^1\times D^1$ be the projection on the first two factors. 
\smallskip

\noindent \textbf{Claim (I)}: the image of $\sigma$ is the open subset
$$V=\{(v,w)\quad |\quad v\, \mbox{ and }\,w\,\, \mbox{ are neither proportional nor}\,\, 0\}.$$
Indeed, notice that the fiber of $\sigma$ over $(v,w)$ with $v,w\neq 0$ is
$$\sigma^{-1}(v,w)=(v,w)\times\{f\in \tilde{U}\,|\, vw(f)=0\}.$$
If $v$ and $w$ are proportional and $f\in \sigma^{-1}(v,w)$, we would have $[f]\in S^3$ such that $V(f)$ is smooth and $v^2(f)=0$. This is impossible by Lemma \ref{LEM:MAGICHESS} so the image of $\sigma$ is in $V$. If $v$ and $w$ are not proportional, one can change coordinates and assume $(v,w)=(y_0,y_1)$. Then, the Fermat cubic $\sum_{i}x_i^3$ is in the fiber which is, therefore, not empty. 
\smallskip

\noindent \textbf{Claim (II)}: for all $p\in \tilde{W}$ the differential $d_p\sigma$ is surjective and $\Ker(d_p\sigma)$ has constant dimension.

\noindent Indeed, for any $v',w'\in D^1$, we have 
$$(d_p\sigma)^{-1}((v',w'))=(v',w')\times \{f'\in S^3\quad|\quad (v',w',f')\, \mbox{satisfies }(\star)\}.$$
Notice that, for any $v,w\neq 0$, the sequence
$$0\to \Ann_{S^3}(vw)\to S^3\xrightarrow{vw} S^1\to 0$$
is exact.
Then, since $(vw'+v'w)(f)$ is determined by the data $p=(v,w,f)$ and $(v',w')$, the surjectivity of $vw$ in the above sequence gives the surjectivity of $d_p\sigma$. The exactness of the above sequence also implies that $\dim(\Ann_{S^3}(vw))$ is constant. By Equation \eqref{EQ:TanZarW} we have
$$\Ker(d_p\sigma)=(0,0)\times \{f'\in S^3\quad|\quad vw(f')=0\}=(0,0)\times \Ann_{S^3}(vw)$$
so we have proved the claim.
\smallskip

Summing up, $\sigma:\tilde{W}\to V$ is s surjective morphism which is submersive and with $\Ker(d_p(\sigma))$ which has constant dimension. Since the target $V$ is smooth, we have that the dimension of $T_p\tilde{W}$ is then constant so $\tilde{W}$ is smooth as claimed.
\smallskip

Consider now the projection $\pi_3:W\to U\subset\bP(S^3)$. Observe that the fiber of $\pi_3$ over $[f]\in U$ is exactly the variety $\Gamma_{f}$ associated to $[f]\in U$ so $\pi_3$ is surjective.
Hence, by generic smoothness we have that the general fiber of $\pi_3$ is smooth, i.e. $\Gamma_{f}$ is smooth, for $[f]\in U$ general.
\end{proof}


\section{Smoothness of $\cD_{k}(f)\setminus \cD_{k-1}(f)$ for $[f]\in U$ general}
\label{SEC:stratsmooth}

In this section we want to prove, under suitable assumptions, that for $[f]\in U=\bP(S^3)\setminus \cC_{sing}$ general, $\cD_{k}\setminus \cD_{k-1}$ is smooth and has the expected dimension. The approach will be similar to the one used in Section \ref{SEC:Singanddesing} in order to prove Theorem \ref{THM:ResolutionOfHessian}.
\smallskip

As recalled in Lemma \ref{LEM:proprQk}, for any $1\leq k\leq n$, the variety $\cQ_k$ parametrizing quadrics in $\bP^n$ of rank at most $k$ is singular exactly along $\cQ_{k-1}$. Hence, $\cQ_{k}^s=\cQ_{k}\setminus \cQ_{k-1}$ is the smooth locus of $\cQ_k$.
Then,
$$\tilde{\cQ}_k=\{q\in S^2\setminus \{0\}\, | \, \Rank(q)\leq k\}\qquad\mbox { and }\qquad \tilde{\cQ}_k^s=\{q\in \tilde{\cQ}_k\, | \, \Rank(q)= k\}$$
are the affine cones (with the origin removed) of $\cQ_k$ and $\cQ_k^s$ respectively and $\tilde{\cQ}_k^s$ is smooth.
Set
$$\cJ_k=\{(q,v,f)\in \tilde{\cQ}_k\times D^1\times S^3\quad |\quad v,f\neq0 \quad \mbox{ and }\quad v(f)=q\}$$
and $\cJ_k^s=\{(q,v,f)\in \cJ_k\, |\, q\in \tilde{\cQ}_k^s\}$. 
\begin{theorem}
\label{THM:smooth2}
Set $c_{n,k}=\codim_{\bP(S^2)}(\cQ_k)=\binom{n-k+2}{2}$. For all $1\leq k\leq n$, the following hold:
\begin{description}
    \item[(a)] $\cJ_k$ and $\cJ_k^s$ are irreducible of dimension $\dim(S^3)+n-c_{n,k}+1$;
    \item[(b)] $\cJ_k^s$ is smooth.
\end{description}
Set $\cD_k^s(f)=\cD_{k}(f)\setminus\cD_{k-1}(f)$ for $[f]\in U$. When $[f]$ is general then either $\cD_k^s(f)=\emptyset$ or the following hold:
\begin{description}
    \item[(c)] $\cD_k^s(f)$ is smooth, i.e. $\Sing(\cD_{k}(f))=\cD_{k-1}(f)$;
    \item[(d)] We have that $\codim_{\bP(S^2)}(\cQ_k)=\codim_{\bP^n}(\cD_k(f))=c_{n,k}$. In particular, $\cD_k(f)$ has the expected dimension;
    \item[(e)] $\deg(\cD_k(f))=\deg(\cQ_k)=\prod_{t=0}^{n-k}\binom{n+t+1}{n-k-t+1}/\binom{2t+1}{t}$.
\end{description}
\end{theorem}

\begin{proof}
First of all, consider the projection
$$\sigma:\cJ_k\to \tilde{\cQ}_k\times(D^1\setminus \{0\})$$ 
and observe that for any $v\neq 0$ we have an exact sequence
\begin{equation}
\label{EQ:Exact_pav}
0\to \Ann_{S^3}(v)\to S^3\xrightarrow{v} S^2\to 0
\end{equation}
since, given $(q,v)$ with $q\in S^2$, $v\neq 0$, one can simply "integrate" $q$ in the direction of $v$ in order to get an $f$ such that $v(f)=q$.
In particular, $a_{n,k}=\dim(\Ann_{S^3}(v))=\dim(S^3)-\dim(S^2)$ does not depend on $v$.
\smallskip

\noindent \textbf{Claim (a)}: Being $v\in D^1$ linear, the fiber of $\sigma$ over $(q,v)$ is
$$\sigma^{-1}(q,v)=(q,v)\times \{f\, |\, v(f)=q\}=(q,v)\times (f_0+\Ann_{S^3}(v))\setminus\{0\}$$
where $f_0$ is any fixed primitive of $q$ in the direction of $v$. In particular $\sigma^{-1}(q,v)$ is an affine space of dimension $a_{n,k}$, possibly with its origin removed. Thus, since the fibers of $\sigma$ are irreducible and equidimensional and the target is irreducible (of dimension $\dim(S^2)-c_{n,k}+n+1$), we can conclude that $\cJ_k$ is irreducible too and has the desired dimension. Moreover, the same holds if we restrict our attention to $\cJ_k^s$ and use the same argument.
\smallskip

\noindent \textbf{Claim (b)}: It is clear from the above discussion that $\sigma$ is surjective. Let $p=(q,v,f)$ be any point in $\cJ_k^s$. The Zariski tangent space to $\cJ_k$ is a subspace of $$T_p(\tilde{\cQ}_k^s\times D^1\times S^3)=T_q(\tilde{\cQ}_k^s)\times D^1\times S^3$$ and
$(q',v',f')\in T_{p}\cJ_k$ if and only if 
$(v+tv')(f+tf')=q+tq' \mod t^2$, i.e. if and only if
$$(\star): \qquad v'(f)+v(f')=q'$$
holds. The differential $d_p\sigma$ is the map sending $(q',v',f')$ to $(q',v')$. 
By the description of the Zariski tangent space it is easy to see that
$$(d_p\sigma)^{-1}(q',v')=(q',v')\times \{f'\in S^3\, |\, v(f')=q'-v'(f)\}\qquad \mbox{ and }\qquad \Ker(d_p\sigma)\simeq\Ann_{S^3}(v)$$
so from exact sequence \eqref{EQ:Exact_pav} we have that
$d_p\sigma$ is surjective and $\dim(\Ker(d_p\sigma))=a_{n,k}$ is constant.
\smallskip

Hence, $\sigma|_{\cJ_k^s}:\cJ_k^s\to \tilde{\cQ}_k^s\times (D^1\setminus\{0\})$ is a surjective and submersive map on a smooth target and its differential has kernel of constant dimension $a_{n,k}$. Hence, the Zariski tangent space of $\cJ_k^s$ has constant dimension equal to $a_{n,k}+\dim(\tilde{\cQ}_k^s)+\dim(D^1)=\dim(S^3)+n-c_{n,k}+1=\dim(\cJ_k^s)$. In particular, $\cJ_k^s$ is smooth as claimed.
\smallskip

From now on, let us assume that $\cD_{k}^s(f)=\cD_{k}(f)\setminus\cD_{k-1}(f)\neq \emptyset$ for $[f]\in U$ general.
\smallskip

\noindent \textbf{Claim (c)+(d)}: Consider the map $\pi_3:\cJ_k^s\to S^3\setminus\{0\}$. The fiber of $\pi_3$ over $f\in S^3\setminus\{0\}$, is
$$R_f=\pi_3^{-1}(f)=\{(v(f),v)\quad |\quad v\neq0,\quad v(f)\in \tilde{\cQ}_k^s\}\times\{f\}$$
so, by the above assumption, $\pi_3$ is dominant.
By (b) $\cJ_k^s$ is smooth so, by generic smoothness, the general fiber of $\pi_3$ is smooth too. In particular, for $[f]\in U$ general we have that $R_f=\pi_3^{-1}(f)$ is smooth of dimension $n-c_{n,k}+1$ (by (a)). Fix a general $[f]\in U$ and consider the restriction $\beta:R_f\to D^1\setminus\{0\}$ of $\pi_2$ to $R_f$. We claim that $\beta$ is an embedding and that $\beta(R_f)$ is the affine cone of $\cD_k^s(f)$ with the origin removed. 
\smallskip

Assume that $\beta(v(f),v)=\beta(w(f),w)$. Then $v=w$ so $\beta$ is injective.
The tangent in $p=(q,v)$ to $R_f$ is given by the vectors
$(q',v')$ such that $q'\in T_{q}\tilde{\cQ}_k^s$, $v'\in D^1$ and 
$v'(f)=q'$ (see $(\star)$). If $d_{p}\beta(q',v')=0$ then $(q',v')$ is such that $v'=0$ and then $q'=0$. In particular, $\beta$ is an embedding.
Hence $\beta(R_f)$ is smooth of dimension $n-c_{n,k}+1$.
On the other hand, 
$$\beta(R_f)=\{v\in D^1\setminus\{0\}\quad | \quad v(f)\in \cQ_k^s\}$$
is the affine cone of $\cD_k^s(f)$ with the origin removed so $\cD_k^s(f)$ is smooth and has dimension $n-c_{n,k}$. In particular, $\Sing(\cD_k(f))=\cD_{k-1}(f)$.
\smallskip

\noindent \textbf{Claim (e)}: This follows by Lemma \ref{LEM:proprQk} since $\cD_k(f)=\bP(J_2)\cap \cQ_k$.
\end{proof}

\begin{corollary}
\label{COR:DIMSING}
Assume that $\cD_k(f)$ is non-empty for $[f]\in U$ general. Then $\cD_{k}(f)\setminus\cD_{k-1}(f)$ is non-empty too for $[f]\in U$ general. In particular, for $[f]\in U$ general, $\cD_k(f)$ is reduced, has the expected dimension and $\Sing(\cD_k(f))=\cD_{k-1}(f)$.
\end{corollary}
\begin{proof}
Let $h=\mathrm{min}\{j \quad | \quad \cD_j(f)\neq\emptyset \, \mbox{ for general } [f]\in U\}$ so $h\leq k$. The claim is clearly true for $k=h$ so for general $[f]\in U$, $\cD_{h}(f)$ is smooth and has the expected dimension. Let us then take $h+1$. Then
$$\dim(\cD_{h}(f))=\Edim(\cD_{h}(f))<\Edim(\cD_{h+1}(f))\leq \dim(\cD_{h+1}(f))$$
where the strict inequality in the middle follows from Remark \ref{REM:EXPDIM}. Hence $\cD_{h+1}(f)\setminus\cD_{h}(f)\neq \emptyset$ for $[f]$ general. By recursion one has the thesis.
\end{proof}

\begin{remark}
\label{REM:FailThmsmooth2}
Notice that, unless $k=n$ (see Theorem \ref{THM:SingH}), we don't have $\Sing(\cD_k(f))=\cD_{k-1}(f)$ for all $[f]\in U$. Indeed, for example, if we consider the Klein cubic fourfold (i.e. $f=x_5^2x_0+\sum_{i=0}^4x_i^2x_{i+1}$) we have that $V(f)$ is smooth, $\cD_{4}(f)\setminus \cD_{3}(f)$ is not empty but it is not smooth. The same holds for the expected dimension: for example, for all $n\geq 2$, the dimension of the Hessian locus of the Fermat cubic is $1$ more than the expected dimension.
\end{remark}




\section{Degeneracy Loci}
\label{SEC:Degeneracyloci}

For this section and the next one, we assume that $\bK=\bC$ (although most of the results hold more generally). In this section, we will study these loci $\cD_k(f)$ from a different, although natural, perspective. In particular, one can think of these varieties as of degeneracy loci of a specific vector bundle map. For a general treatment of this subject, the reader can refer to \cite{Tu85} and \cite{Laz}. 
\smallskip

Let $X$ be a smooth variety of dimension $n$. We are interested in degeneracy loci of a symmetric morphism between vector bundles (symmetric over each fiber) on $X$, i.e. a morphism
$$\varphi:E\rightarrow E^*\otimes L \quad \mbox{such that }\quad  \varphi=\varphi^*\otimes \mathrm{Id}_L,$$
where $E$ and $L$ are a vector bundle of rank $e$ and a line bundle over $X$ respectively.
Then one can define the degeneracy loci at order $k$ for such a map as
$$\cD'_k(\varphi)=\{x\in X \ | \ \Rank(\varphi_x)\le k\}.$$

Let us summarize in the following Theorem, some important results about the non-emptiness and the connectedness of these degeneracy loci (see \cite{FL_Positive}, \cite{HT_CONN_OddRanks} and \cite{Tu_Even}). 

\begin{theorem}
\label{THM:degenloci}
Let $\varphi:E\to E^*\otimes L$ be a symmetric morphism of vector bundles of rank $e$ and consider $k\leq e$. If $n\geq\binom{e-k+1}{2}$, then $\cD'_k(\varphi)$ is non-empty.\\
Moreover, if $(\mathrm{Sym}^2E^*)\otimes L$ is ample, the following hold:
\begin{enumerate}[(a)]
    \item if $k$ is even and $n-\binom{e-k+1}{2}\geq 1$, then $\cD'_k(\varphi)$ is connected;
    \item if $k$ is odd and $n-\binom{e-k+1}{2}\geq e-k$, then $\cD'_k(\varphi)$ is connected.
\end{enumerate}
\end{theorem}

Under the assumptions of the above theorem, it is also well known that for general $\varphi$, $\cD_k'(\varphi)$ is either empty or it has as (pure) dimension the expected one. Moreover, $\Sing(\cD_k'(\varphi))=\cD_{k-1}'(\varphi)$. 

We will be interested in considering a symmetric matrix $M$ of order $n+1$ with coefficients in $S^1$ and the symmetric homomorphism of vector bundles induced by $M$, namely
\begin{equation}
\label{EQ:SYMMORPH}
\varphi_M:\cO_{\bP^n}^{n+1}\xrightarrow{M\cdot } \cO_{\bP^n}^{n+1}(1).
\end{equation}

If we further specialize to the case where $M$ is the Hessian matrix $H_f$ of a cubic form $[f]\in\bP(S^3)$, we have that the locus $\cD_k(f)$ introduced in Section \ref{SEC:Notations} is exactly the degeneracy locus $\cD'_k(\varphi_{H_f})$ introduced above. Hence, by the results showed in the previous sections, what follows refers also to the Hessian matrix associated to a general smooth cubic hypersurface.
Let us stress that the symmetric morphism induced by the Hessian matrix of a cubic form is not general among symmetric morphisms as in equation \eqref{EQ:SYMMORPH}. Results that hold for such general morphisms are not then automatically true for the general Hessian matrix: these need to be adapted with ad hoc techniques, as done, for example, in \cite{RV} where it is proved that for $n=5$, $\cD_3'(\varphi_{H_f})=\emptyset$ for $[f]\in \bP(S^3)$ general, or also in the previous sections (see Theorems \ref{THM:SingH} and \ref{THM:smooth2}).

By using \ref{THM:degenloci} and the above mentioned results, we then have

\begin{proposition}
\label{PROP:Ysmooth}
The following hold:
\begin{enumerate}[(a)]
    \item if $n\geq3$, $\Sing(\cH_f)$ is non-empty for all $[f]\in\bP(S^3)$. Moreover, if $n\geq 4$, it is also connected.
    \item if $3\leq n\leq 5$, for $[f]\in U$ general,  $\Sing(\cH_f)$ is smooth of dimension $n-3$.
\end{enumerate}
\end{proposition}
\begin{proof}
Claim (a): the non-emptyness of $\cD_{n-1}(f)$ follows directly from Theorem \ref{THM:degenloci} or from the argument about the expected dimension of Remark \ref{REM:EXPDIM}. The connectedness for $n\geq 5$ follows again from Theorem \ref{THM:degenloci} while the case $n=4$ has been treated in \cite{AR}. For claim (b), if $n=3,4$ the result is already known (see Remark \ref{RMK:example}) so we can assume $n=5$. In this case, by (a) we have that $\cD_4\neq \emptyset$. On the other hand, by \cite[Lemma 2.1]{RV} we have that $\cD_3=\emptyset$ when $[f]$ is general. Then, the claim follows from Theorem \ref{THM:smooth2}.
\end{proof}

\begin{remark}
\label{REM:INTERESTINGCASE}
The case $n=5$ is the last one where the singular locus of $\cH$ is smooth generically.
Indeed, we know that $\cD_{n-2}$ is contained in the singular locus of $\Sing(\cH)$. By Theorem \ref{THM:degenloci} we can easily see that in this case the condition $n\geq\binom{n+1-(n-2)+1}{2}=6$ tells us that for every $n\geq6$ the singular locus of the Hessian hypersurface associated to any $[f]\in\bP(S^3)$ is itself singular. On the other hand, this case is also the first one which has still to be analysed (by Remark \ref{RMK:example}. We will focus on this specific case in the Sections \ref{SEC:SingHessCubicFourfold}.
\end{remark}

Let us now consider again the map $\varphi_M=M\cdot$ where $M$ is symmetric of order $n+1$ with coefficients in $S^1$ as above. By following the strategy presented with a more general flavour in \cite{HT_CH}, one can compute the odd Chern classes of $Y=\cD_k'(\varphi_M)$ assuming that $Y$ is smooth. In this case, indeed one has that $\cD'_{k-1}(\varphi_{M})$ is empty. We can then consider the following exact sequence on $Y$
\begin{equation}
\label{EQ:EXSEQChern}
0\rightarrow B\rightarrow\cO_Y^{n+1}\xrightarrow{\alpha}\cO_Y^{n+1}(1)\rightarrow C\rightarrow0
\end{equation}
where $\alpha$ is the restriction to $Y$ of $\varphi_M$, and $B=\ker(\alpha)$ and $C=\coker(\alpha)$ are locally free (on $Y$) of rank $n+1-k$.
Since $\alpha$ is symmetric, dualizing and tensoring with $\cO_Y(1)$ we get the relation $$C=B^*\otimes \cO_Y(1).$$

Starting from this we derive an explicit relation satisfied by the canonical divisor $K_Y$ which will be used in Section \ref{SEC:SingHessCubicFourfold}.

\begin{proposition}
\label{PROP:canonico}
Assume that $Y=\cD'_k(\varphi_M)$ is smooth. Then, we have
$$2K_{Y}=(n+1)(n-k)H|_{Y}$$
where $H$ denotes the hyperplane class of $\bP^n$. 
\end{proposition}
\begin{proof}

As $\cN_{Y/\bP^n}=(\Sym^2B^*)\otimes\cO_{Y}(1)$ (see \cite{HT_CH}), 
we get
$$c_1(\cN_{Y/\bP^n})=\Rank(\Sym^2B^*)c_1(\cO_{Y}(1))+c_1(\Sym^2B^*)=\binom{n+2-k}{2}H|_{Y}-(n+2-k)c_1(B).$$
Since $c_1(Y)=c_1(\bP^n)|_{Y}-c_1(\cN_{Y/\bP^n})$ by the normal exact sequence we have
$$K_Y=-c_1(Y)=-(n+1)H|_Y+\binom{n+2-k}{2}H|_Y-(n+2-k)c_1(B).$$
Since $C=B^*\otimes \cO_Y(1)$, we obtain $c_1(C)=(n+1-k)H|_Y-c_1(B)$. From this relation and from the exact sequence \eqref{EQ:EXSEQChern}, one easily gets that $2c_1(B)=-kH|_Y$ which allows us to conclude.
\end{proof}

\begin{remark}
Let us notice that Propositions \ref{PROP:Ysmooth}, \ref{PROP:canonico} and Theorem \ref{THM:smooth2} give a description of $\Sing(\cH)$, where $\cH$ is the Hessian hypersurface in $\bP^4$ associated to a general cubic threefold. Indeed, one gets that it is a smooth, irreducible curve of degree $20$ and genus $26$, as shown in \cite{AR} with different techniques.
\end{remark}

In a similar way, if $X=V(f)\subset \bP^n$ is a general smooth cubic hypersurface, one can study the {\it smallest} non-empty $\cD_k(f)$, namely, the one for which $\cD_{k+1}(f)=\emptyset$.

\begin{remark}
It is easy to see that the {\it smallest} $\cD_k(f)$ is a curve exactly when $n$ and $k$ are such that 
$$n=n_s=\binom{s+1}{2}+1,\qquad\mbox{ and }\qquad k=k_s=\binom{s}{2}+2$$
with $s\geq 1$. The first two cases, namely $(2,2)$ and $(4,3)$, represent, respectively, the case of cubic curves (for which $\cD_2(f)=\cH_f$ is again a smooth cubic curve) and the case studied by Adler (i.e. $\cD_3(f)=\Sing(\cH_f)$ for $V(f)$ general smooth cubic threefold). By setting $C_s=\cD_{k_s}(f)$, by Proposition \ref{PROP:canonico} one has that 
$$2K_{C_s}=\left(\binom{s+1}{2}+2\right)(s-1)H|_{C_s}$$
so, using Theorem \ref{THM:smooth2}, one easily obtains the degree of $C_s$ and its genus. 
\end{remark}

In the next section we will deal with the case $n=5$ and describe the {\it smallest} locus, namely the surface $Y$ arising as singular locus of the Hessian variety of a general cubic fourfold. In order to do this we will need a construction related to degeneracy loci and covers, that we will develop in the next subsection and that is inspired by a natural question coming from Proposition \ref{PROP:canonico}: one might wonder whether in this case $K_Y=3H|_Y$ or not since we know that $2K_Y=6H|_{Y}$.


\subsection{Covers and connectedness}
\label{SUBSEC:DoubleCovers}
$\,$\\
In this subsection we present a general construction that allows us to describe the existence of $2:1$ covers for suitable degeneracy loci of symmetric maps between vector bundles.

Let us start by considering a vector bundle $E$ of rank $e+1$ on an irreducible projective variety $X$ of dimension $n$ and a symmetric map $\varphi:E\rightarrow E^*\otimes L$ where $L$ is a line bundle. For any $m$ with $1\leq m\leq e+1$ one can consider the relative Grassmannian $\pi:\bG=G(m,E)\rightarrow X$ associated to $E$, a fiber bundle whose fiber over $x$ is the Grassmannian $G(m,E_x)$ of $m$-dimensional subspaces of $E_x$.
We will denote by $\cS_E$ and $\cU_E$ the tautological bundle and the universal bundle of $\bG$ respectively, which fit into the exact sequence
\begin{equation}
\label{EQ:TAUTOLOGICAL}
\xymatrix@R=30pt{
0 \ar[r] &  
    \cS_E \ar[r]^-{\alpha}  &
    \pi^*E \ar[r]  &
    \cU_E \ar[r] &
    0
}
\end{equation}

where $\alpha_{W}:(\cS_E)_{W}\simeq W\to (\pi^*E)_W\simeq E_{\pi(W)}$
is the natural inclusion of the subspace $W\subseteq E_{\pi(W)}$ for any $W\in \bG$. 
\medskip

Denote by $\alpha^*\otimes\id_{\pi^*L}:\pi^*(E^*\otimes L) \to \cS_E^*\otimes \pi^*(L)$  the map obtained by dualizing $\alpha$ and then tensoring by $\pi^{*}(L)$. Then from the diagram
$$
\xymatrix@C=45pt@R=15pt{
0 \ar[r] &
    \cU_E^*\otimes \pi^*L \ar[r] &
    \pi^*(E^*\otimes L) \ar[r]^-{\alpha^*\otimes \id_{\pi^*(L)}} &
    \cS_E^*\otimes \pi^*(L) \ar[r] &
    0 \\
  & &  \colorbox{white}{\phantom{XXXX}}\ar@{->}[u]^-{\pi^*(\varphi)} & &   \\
0 \ar[r] &  
    \cS_E \ar[r]_-{\alpha} \ar`u[u]`[rruu]_(0.7){\psi}[rruu] &
    \pi^*E \ar[r]\ar@{-}[u]  &
    \cU_E \ar[r] &
    0
}
$$
one can define the morphism $\psi:\cS_E\rightarrow \cS_E^*\otimes\pi^*(L)$.

\begin{remark}
\label{REM:psi}
As $\varphi$ is symmetric, we have $(\varphi)^*\otimes \id_{\pi^*(L)}=\varphi$, so $\pi^*((\varphi)^*\otimes \id_{\pi^*(L)})=\pi^*\varphi$. Hence
$$\psi^*\otimes \id_{\pi^*L}=(\alpha^*\otimes \id_{\pi^*(L)})\circ((\pi^*(\varphi))^*\otimes \id_{\pi^*(L)})\circ \alpha=(\alpha^*\otimes \id_{\pi^*(L)}) \circ\pi^*(\varphi)\circ \alpha=\psi$$
so $\psi$ is a symmetric morphism of vector bundles on $\bG$. Notice, moreover, that we can interpret $\psi$ as a section of the bundle $\Hom^{sym}(\cS_E, \cS_E^*\otimes\pi^*(L))=\Sym^2(\cS_E^*)\otimes \pi^*(L)$.
\end{remark}

We fix now an integer $k$ such that the degeneracy locus $\cD_k'(\varphi)=\cD_k'$ is non-empty. If $x\in \cD_k'$ then $\varphi_x:E_x\to E_x^*\otimes L_x\simeq E_x^*$ is a symmetric morphism. Then, $\varphi_x$ can be thought either as a symmetric bilinear form or as a polynomial of degree $2$ which identifies a quadric $Q_x\subseteq \bP^e$. By construction, $Q_x$ is a quadric of rank at most $k$. Recall that a linear subspace in $\bP^e$ is called isotropic with respect to $Q_x$ if it is contained in $Q_x$.
\smallskip

\begin{remark}
\label{REM:ISOTROPIC}
Assume that $Q_x$ has rank exactly $k$. Then $Q_x$ contains either one or two families of maximal isotropic subspaces. More precisely, $Q_x$ always contains:
\begin{itemize}
    \item two irreducible $\binom{h}{2}$-dimensional families of maximal isotropic $(e-h)$-planes if $k=2h$;     
    \item one irreducible $\binom{h+1}{2}$-dimensional family of maximal isotropic $(e-h-1)$-planes if $k=2h+1$.
\end{itemize}
In order to see this, notice that the vertex of $Q_x$ is a $(e-k)$-plane so, by cutting $Q_x$ with a general $(k-1)$-plane $\Lambda$, one has a smooth quadric of dimension equal to $k-2$. Then $Q_x\cap \Lambda$ contains either one family of dimension $\binom{h+1}{2}$ or two families of dimension $\binom{h}{2}$ of $(h-1)$-planes depending on the parity of $k$ (see \cite[Chapter 6, page 735]{GH}). One then concludes by observing that families of linear spaces in $Q_x\cap\Lambda$ and in $Q_x$ are in bijection via joint union with the vertex of $Q_x$.
If the rank of $Q_x$ is strictly less than $k$, with the same method, one can show that $Q_x$ always contains $(e-h)$-planes when $k=2h$ or $(e-h-1)$-planes when $k=2h+1$.
\end{remark}

From now on set 
$$m=\begin{cases}
e-h & \quad\mbox{ if }k=2h+1 \\
e-h+1 & \quad\mbox{ if }k=2h 
\end{cases}
$$
so that, from Remark \ref{REM:ISOTROPIC}, the quadric $Q_x$ contains (possibly non-maximal) isotropic subspaces of dimension $m-1$ for all $x\in \cD_k'$. 

\begin{lemma}\label{LEM:Isotropic}
Let $W\in \bG=G(m,E)$ and set $x=\pi(W)$, i.e. $W\subseteq E_x$ is a $m$-dimensional linear subspace of $E_x$. Then $\bP(W)$ is an isotropic subspace for $Q_x$ if and only if $\psi_W\equiv 0$.
\end{lemma}

\begin{proof}
Recall that $\alpha_W:W\to E_x$ is the inclusion of $W$ in $E_x$. Hence, $(\alpha^*)_W=(\alpha_W)^*:E_x^*\to W^*$ takes a linear form on $E_x$ and restrict it to $W$. Since $\pi^*(\varphi)_W:E_x\to E_x^*\otimes L_x\simeq E_x^*$ is the map sending $v\in E_x$ to $\varphi_x(v)$ one has that 
$$\forall W\in \bG(m,E)\, \mbox{ and }\, \forall\, v\in(\cS_E)_W=W\qquad  \psi_W(v)=\varphi_x(v)|_W,$$
i.e. the linear form $\varphi_x(v)$  on $E_x$ restricted to $W$. 
Since $\bP(W)$ is an isotropic subspace for $Q_x$ if and only if $\forall v,v'\in W$ one has $\varphi_x(v)(v')=0$, we have that $\bP(W)$ is isotropic for $Q_x$ if and only if $\psi_W\equiv 0$.
\end{proof}

Let us now denote the zero locus of the symmetric morphism $\psi$ by $Z=Z(\psi)=\{W\in \bG \,|\, \psi_W\equiv 0\}$. 
\begin{remark}
One can compute the expected dimension of $Z$ as
$$\Edim(Z)=\dim(\bG)-\Rank(\Sym^2(\cS_E^*)\otimes\pi^*(L))=\dim(X)+m(e+1-m)-\binom{m+1}{2}$$
since $Z=Z(\psi)$, where $\psi$ is section of the bundle $\Sym^2(\cS_E^*)\otimes\pi^*(L)$ and $\Rank(\cS_E^*)=m$.
\end{remark}

By construction and, in particular, by Lemma \ref{LEM:Isotropic}, it is then clear that $\cD_k'=\cD_k'(\varphi)\subseteq\pi(Z)$. Indeed, to any point $x\in\cD_k'$ corresponds a quadric $Q_x$ which contains at least one family of isotropic $(m-1)$-planes: for every such plane $W$, we have that $\psi_W\equiv0$ and $\pi(W)=x$. 
\smallskip

Let us now consider the following easy result of linear algebra:
\begin{lemma}
Let $W\subseteq V$ be two vector spaces of dimension $l$ and $e+1$ respectively, $\iota:W\to V$ the natural inclusion and let $\varphi:V\to V^*$ be a linear map.  If $\eta=\iota^*\circ \varphi\circ \iota:W\to W^*$ is the zero map, then the rank of $\varphi$ is at most $2(e+1-l)$.
\end{lemma}

\begin{proof}
For brevity, set $\tilde{\varphi}=\varphi\circ \iota$. Since $\eta\equiv0$, we have that $\mathrm{Im}(\tilde{\varphi})\subseteq\Ker(\iota^*)$, which has dimension $e+1-l$, and so $\dim(\mathrm{Im}(\tilde{\varphi}))\le e+1-l$ and also $\dim(\Ker(\tilde{\varphi}))\ge l-(e+1-l)=2l-e-1$. Since $\iota$ is injective, we also have that $\dim(\iota(\Ker(\tilde{\varphi})))\ge 2l-e-1$, but $\iota(\Ker(\tilde{\varphi}))=W\cap\Ker(\varphi)\subseteq\Ker(\varphi)$. From this we get that $\Ker(\varphi)$ has dimension at least $2l-e-1$ and so $\Rank(\varphi)\le (e+1)-(2l-e-1)$, so we get the claim.
\end{proof}

In particular, for $k=2h$ even, by setting $l=e-h+1$ one also has that $\pi(Z(\psi))\subseteq \cD_k'(\varphi)$:

\begin{corollary}
Assume that $k=2h$ is even. Then $\pi(Z(\psi))=\cD_k'(\varphi)$.
\end{corollary} 

Assume, as above, that $k=2h$ is such that the degeneracy locus $\cD_k'(\varphi)=\cD_k'$ is non-empty and denote by $\pi:Z(\psi)\to \cD_k'(\varphi)$ also the restriction of $\pi:\bG\to \bP^e$ and consider $m=e-h+1$. The fiber of $\pi$ over $[x]\in \cD_k'$ is, by construction, a variety parametrizing the isotropic $(m-1)$-planes in $Q_x$. If $[x]\in \cD_k'\setminus\cD_{k-1}'$, $\pi^{-1}([x])$ parametrizes maximal isotropic $(m-1)$-planes in $Q_x$ so it has two irreducible disjoint components of dimension $\binom{h}{2}$ by Remark \ref{REM:ISOTROPIC}. One can consider the Stein factorization of $\pi$, i.e.
\begin{equation}
\label{EQ:DiagDoubleCover}
\xymatrix@R=15pt@C=40pt{
& Z=Z(\psi)\ar@{^{(}->}[r]\ar[ld]_-{\alpha}\ar@{->>}[dd]^-{\pi} & \bG\ar@{->>}[dd]^-{\pi} \\
\tilde{Y}\ar[rd]_-{\beta}\\
& Y=\cD_k'(\varphi)\ar@{^{(}->}[r] & \bP^e
}
\end{equation}
where $\alpha$ has connected fibers and $\beta$ is finite. From the above discussion, one can see that the map $\beta$ is a $2:1$ morphism, whose possible ramification lies in $\beta^{-1}(\cD_{k-1}')$.
\smallskip

It is then interesting to focus on the case where $\cD_{k-1}'$ is empty (this happens, for example, when $\cD_k'$ is smooth) and $Y=\cD_k'$ is connected: in this situation one can wonder whether the finite map $\beta$, which is then an unramified $2:1$ cover, is non-trivial. This covering is trivial if and only if $\tilde{Y}$ is not connected, i.e. if and only if $Z$ is not connected since $\alpha$ has connected fibers. Notice that, even if we assume that $E$ and $k$ satisfy the hypotheses of Theorem \ref{THM:degenloci}(a) we can't guarantee the connectedness of $Z=Z(\psi)=\cD_0'(\psi)$ since $(\Sym^2\cS_E^*)\otimes \pi^*L$ is not ample in general.
\smallskip

Let us now propose a sufficient condition which allows us to obtain the connectedness of $Z$, under suitable assumptions that will be satisfied in the case we will consider.

\begin{lemma}
\label{LEM:CohOfResolution}
Let 
$$
0\to \cF_p\xrightarrow{\lambda_{p-1}} \cdots \xrightarrow{\lambda_{2}} \cF_2 \xrightarrow{\lambda_1}\cF_1 \xrightarrow{\lambda_0} \cF_0 \to 0$$
be an exact sequence of sheaves. Let $k\geq 0$ and assume that $H^{j+k-1}(\cF_j)=0$ for all $j$ such that $1\leq j \leq p$. Then one has $H^k(\cF_0)=0$.
\end{lemma}

\begin{proof}
Split the starting exact sequence into short exact sequences of the form
$$(\star_j): \qquad 0\rightarrow K_j\rightarrow \cF_j\rightarrow K_{j-1}\rightarrow0$$
for $1\leq j\leq p-1$ with $K_0=\cF_0$ and $K_{p-1}=\cF_p$:
$$
\xymatrix@C15pt@R20pt{
 && K_{p-1}  \ar@{^{(}->}[rd]&&&& K_2 \ar@{^{(}->}[rd] && && K_0\ar[rd]|-{\colorbox{white}{=}}\\
0 \ar[r] &
    \cF_p\ar[rr]^-{\lambda_{p-1}}  \ar[ru]|-{\colorbox{white}{=}}&&
    \cF_{p-1}\ar[rr]^-{\lambda_{p-2}}   \ar@{^{(}->}[rd]&&
    \dots \ar@{..}[ru]\ar[rr]^-{\lambda_{2}}&&
    \cF_2 \ar[rr]^-{\lambda_{1}} \ar@{->>}[rd] && 
    \cF_1 \ar@{->>}[rr]^-{\lambda_{0}}  \ar@{->>}[ru] &&
    \cF_0 \ar[r] &
    0. \\
 &&&& K_{p-2}\ar@{..}[ru]&&&& K_1\ar@{^{(}->}[ru]\\
}
$$

From exact sequence $(\star_{p-1})$ and by assumption we have
$$\cdots\rightarrow\underbrace{H^{p-2+k}(\cF_{p-1})}_{=0}\rightarrow H^{p-2+k}(K_{p-2})\rightarrow\underbrace{H^{p-1+k}(\cF_{p})}_{=0}\rightarrow \cdots$$
so $H^{p-2+k}(K_{p-2})=0$. By a recursive argument we can show that if we have $H^{j+k}(K_{j})=0$ and $H^{j-1+k}(\cF_j)=0$ (the latter is true by assumption), then also $H^{j-1+k}(K_{j-1})=0$ holds. This claim follows easily from the long exact sequence in cohomology induced by $(\star_j)$:
$$\cdots\rightarrow\underbrace{H^{j-1+k}(\cF_{j})}_{=0}\rightarrow H^{j-1+k}(K_{j-1})\rightarrow\underbrace{H^{j+k}(K_{j})}_{=0}\rightarrow \cdots.$$
At the end of this process we get $0=H^k(K_0)=H^k(\cF_0)$ as desired.
\end{proof}

\begin{corollary}
\label{COR:koszul}
Let $T$ be a smooth connected variety and let $Z$ be a subvariety of $T$ of codimension $p$ which is the zero locus of a section $\theta$ of a vector bundle $\cP$ on $T$ of rank $p$. 
\begin{enumerate}[(a)]
    \item For any $M\in \Pic(T)$ and $k\geq 0$, if $H^{j+k-1}(M\otimes \bigwedge^j\cP^*)=0$ for $1\leq j\leq p$, then one has $H^k(\cI_{Z/T}\otimes M)=0$;
    \item If $H^{j}(\bigwedge^j\cP^*)=0$ for $1\leq j\leq p$, then $Z$ is connected;
    \item For any $k\geq 1$, if $h^k(\cO_T)=0$ and $H^{j+k}(\bigwedge^j\cP^*)=0$ for $1\leq j\leq p$, then $h^k(\cO_Z)=0$.
\end{enumerate}
\end{corollary}

\begin{proof}
Since $Z$ is of the expected dimension by assumption, the Koszul sequence induced by $\theta$
$$
\xymatrix@C15pt@R20pt{
0 \ar[r] &
    \bigwedge^p\cP^*\ar[rr]^-{\lambda_{p-1}} &&
    \dots \ar[rr]^-{\lambda_{2}} &&
    \bigwedge^2\cP^* \ar[rr]^-{\lambda_{1}} && 
    \cP^* \ar@{->>}[rr]^-{\lambda_{0}} &&
    \cI_{Z/T} \ar[r]&
    0
}
$$
is exact. By tensoring with $M\in \Pic(T)$ we preserve exactness so $(a)$ follows directly from Lemma \ref{LEM:CohOfResolution}. For $(b)$, consider the exact sequence
$$0\rightarrow\cI_{Z/T}\rightarrow\cO_T\rightarrow\cO_Z\rightarrow0.$$
Since $T$ is connected we have an exact sequence
$$0\to H^0(\cO_T)\to H^0(\cO_Z)\to H^1(\cI_{Z/T})$$
hence the vanishing of $H^1(\cI_{Z/T})$ implies the connectedness of $Z$ so we conclude using $(a)$.
For $(c)$, from the same sequence, we have an injection
$$0=H^k(\cO_T)\to H^k(\cO_Z)\hookrightarrow H^{k+1}(\cI_{Z/T}).$$
We can then conclude since, under our assumption, the second term is zero by $(a)$.
\end{proof}

Hence, assuming that the degeneracy locus $Y=\cD_k'(\varphi)$ is connected and smooth, one can show that $\beta:\tilde{Y}\to Y$ is a non-trivial unramified covering by proving that $H^j(\bigwedge^j \cP^*)=0$ with $\cP=\Sym^2(\cS_E^*)\otimes \pi^*L$ whenever $1\leq j\leq \binom{m+1}{2}$.
\medskip 

\subsection{Application to the degeneracy loci $\cD_4'(\varphi_M)\subseteq \bP^5$}
\label{SUBSEC:Application}
$\,$\\
As an application of the above discussion set $X=\bP^5$ and consider $\varphi=\varphi_M:\cO_{\bP^5}^{6}\to \cO_{\bP^5}(1)^6$ where $M$ is symmetric of order $6$ and coefficients in $S^1$. We are interested in this case as consequence of the considerations made in Remark \ref{REM:INTERESTINGCASE}.
\smallskip

Assume that the surface $Y=\cD_4'(\varphi)$ is smooth so that $\cD_3'(\varphi)=\emptyset$. Here with respect to the notations of Subsection \ref{SUBSEC:DoubleCovers} (and the objects in Diagram \eqref{EQ:DiagDoubleCover}), we have $k=4$ and $h=2$ so $\tilde{Y}$ is a smooth surface which is an unramified double cover of $Y$. The fiber of $\alpha:Z\to \tilde{Y}$ over $p\in \beta^{-1}([x])$ is a $\bP^1$ which parametrizes one of the two irreducible families of maximal isotropic $3$-planes contained in the quadric $Q_x$. 
Hence, $\alpha$ has irreducible and equidimensional fibers so $Z$ is a threefold: we can apply Corollary \ref{COR:koszul} in order to study the connectedness of the covering $\tilde{Y}$. Indeed, we have that $\tilde{Y}$ is connected if $Z$ is so and this is equivalent to $h^0(\cO_Z)=1$.
\smallskip

Notice that since $E=\cO_{\bP^5}^6$, we have $\bG=G(4,\cO_{\bP^5}^6)=G(3,\bP^5)\times \bP^5$ so $h^1(\cO_{\bG})=0$. We denote by $\pi_1$ and $\pi_2$ the two natural projections. If $\cF$ and $\cG$ are sheaves on $G(3,\bP^5)$ and $\bP^5$ respectively, we set $\pi_1^*(\cF)\otimes \pi_2^*(\cG):=\cF\boxtimes\cG$ for brevity. Recall that $Z$ is the zero locus of a section of the vector bundle $\cP=\Sym^2(\cS_{\cO_{\bP^5}^6}^*)\otimes \pi^*\cO_{\bP^5}(1)$ which, in this case, can be written as
$\cP=\Sym^2(\cS^*)\boxtimes \cO_{\bP^5}(1)$
where $\cS$ is the tautological bundle of $G(3,\bP^5)$.
Thus, we have
\begin{equation}
\label{EQ:BigWedge}
\bigwedge^j\cP^*=
\bigwedge^j\left(\Sym^2(\cS)\boxtimes \cO_{\bP^5}(-1)\right)=
\left(\bigwedge^j\Sym^2(\cS)\right)\boxtimes \cO_{\bP^5}(-j)
\end{equation}
so it is clear that the vanishing of the coomology group $H^i(\bigwedge^j\cP)$ strongly depends on the vanishing of some of the cohomology groups of $\bigwedge^j\Sym^2(\cS)$ on the Grassmannian $G(3,\bP^5)$. 

\begin{proposition}
\label{PROP:coomS}
One has $H^{i}(\bigwedge^j\Sym^2\cS)=0$ for all pairs $(i,j)$ with $i\geq 0$, $1\leq j\leq 10$ except for the cases where $(i,j)\in\{(2,2), (2,3), (2,4), (4,5), (4,6), (4,7), (6,9)\}$. For these cases, $H^{i}(\bigwedge^j\Sym^2\cS)\neq 0$.
\end{proposition}

\begin{proof}
It is well known that $G(k,\bP^n)$ is an homogeneous variety which can be described as $\SL(n+1)/P$, where $P$ is a suitable Lie group of triangular matrices in $\SL(n+1)$. 

The proof of the proposition is based on the study of maximal weights from Lie theory and on Bott's theorem on homogeneous vector bundles on Grassmannians. The interested reader can refer to \cite{FH} and \cite{Ott} (or also \cite[Theorem IV']{Bott} for the original statement of the 
theorem). We give here a very brief insight, with the notation used in \cite{Ott}. Recall that homogeneous vector bundles of rank $r$ on $G(k,\bP^n)$ correspond to representations $\rho:P\to \GL(r)$. For example, $\cS$  corresponds to the so-called standard representation of $P$. One can associate to a representation $\rho$ its weights, which have a natural partial ordering. Maximal weights of $\rho$ correspond to direct summands of the associated vector bundle $E_{\rho}$. 
The claim is obtained by an explicit and long computation of all the maximal weights associated to $E_{\rho}=\bigwedge^j\Sym^2(\cS)$ and then, by using Bott's theorem. This allows us to describe the cohomology groups of $E_{\rho}$ by investigating the behaviour of its maximal weights with respect to a suitable root system.
\end{proof}


For all $d\in \bZ$, from Equation \eqref{EQ:BigWedge} and by K\"unnet's Theorem we have
\begin{equation}
H^i\left(\pi_2^*\cO_{\bP^5}(d)\otimes \bigwedge^j\cP^*\right)\simeq \bigoplus_{a+b=i}\left(H^a\left(\bigwedge^j\Sym^2\cS\right)\otimes H^b(\cO_{\bP^5}(d-j))\right).
\end{equation}

In particular, since for $j\geq 1$ the group $H^b(\cO_{\bP^5}(-j))$ is trivial whenever $b\neq 5$, we have
$$H^j\left(\bigwedge^j\cP^*\right)\simeq H^{j-5}\left(\bigwedge^j\Sym^2\cS\right)\otimes H^5(\cO_{\bP^5}(-j))=0$$
by Proposition \ref{PROP:coomS} and since $H^5(\cO_{\bP^5}(-j))\simeq H^{0}(\cO_{\bP^5}(j-6))^*=0$ for $j<6$. Then, using Corollary \ref{COR:koszul}(b) we have that $Z$ is connected and the same holds for $\tilde{Y}$.
With the same reasoning just used, we also get that
$H^{j+1}(\bigwedge^j\cP^*)=0$ for $j\geq 1$ so we can conclude, again by Corollary \ref{COR:koszul}(c), that $h^1(\cO_Z)=0$.
\smallskip

Similarly, for $d\geq 0$ one has
$$H^{j-1}\left(\pi_2^*(\cO_{\bP^5}(d))\otimes \bigwedge^j\cP^*\right)\simeq \begin{cases}
\mbox{if }j\leq d  & H^{j-1}\left(\bigwedge^j\Sym^2\cS\right)\otimes H^0(\cO_{\bP^5}(d-j))\\
\mbox{if }j>d  & H^{j-6}\left(\bigwedge^j\Sym^2\cS\right)\otimes H^5(\cO_{\bP^5}(d-j))\\
\end{cases}$$
which is always trivial if $d\leq 2$ by Proposition \ref{PROP:coomS}. Hence, by Corollary \ref{COR:koszul}(a), one has $H^0(\cI_{Z/\bG}\otimes\pi_2^*\cO_{\bP^5}(d))=0$ for $d=1,2$.
\smallskip

Summing up, one has the following

\begin{proposition}
\label{PROP:RegularityAndUnramifiedCovering}
Let $n=5$. Assume that $M$ is either a general symmetric matrix with coefficients in $S^1$ or $M=H_f$ for $[f]\in \bP(S^3)$ general. Then the variety $Z$ constructed above is a connected threefold with $h^1(\cO_Z)=0$, $h^0(\cI_{Z/\bG}\otimes\pi_2^*\cO_{\bP^5}(1))=0$ and $h^0(\cI_{Z/\bG}\otimes\pi_2^*\cO_{\bP^5}(2))=0$. Moreover, $\tilde{Y}$ is a connected surface so $\beta:\tilde{Y}\to Y$ is a non-trivial unramified double cover.
\end{proposition}

\section{The singular locus of the Hessian variety of a general cubic fourfold}
\label{SEC:SingHessCubicFourfold}

In this last section, by assuming again $\bK=\bC$, we describe the variety $Y=\cD_4(f)=\Sing(\cH_f)$ when $V(f)\subset \bP^5$ is a general smooth cubic fourfold. The same description is achieved with the same techniques also when $Y=\cD_4'(\varphi_M)$ for a general symmetric matrix of order $6$ with entries that are linear forms in $x_0,\dots, x_5$.

The starting point of this analysis is based on results obtained more generally in the previous sections which we sum up in this lemma:

\begin{lemma}
\label{LEM:BasicFactsOnY}
Assume that $[f]\in U$ is general and denote by $H$ the hyperplane class of $\bP^5$. We have the following:
\begin{enumerate}[(a)]
    \item $\cH_f=\cD_5(f)$ has singular locus $Y=\Sing(\cH_f)=\cD_4(f)$ which has dimension $2$;
    \item $Y$ is connected, smooth and it is cut by $21$ quintic hypersurfaces;
    \item as subvariety of $\bP^5$, $Y$ has degree $35$;
    \item there exists $\eta\in \Pic^0(Y)[2]$ such that $3H|_Y+\eta=K_Y$.
\end{enumerate}
In particular, $Y$ is a minimal surface of general type.
\end{lemma}

\begin{proof}
Most of the information for $(a),(b)$ and $(c)$ follows from Theorem \ref{THM:smooth2} and Propositions \ref{PROP:Ysmooth}. The surface $Y$ is the degeneracy locus of rank $4$ of the Hessian matrix $H_f$, which is a symmetric matrix of order $6$, so it is defined by the vanishing of the $36$ minors of $H_f$ of order $5$. Each minor gives a quintic equation and, by symmetry, it is enough to consider only $21$ of them. For $(d)$, by Proposition \ref{PROP:canonico}, we have $2K_Y=6H|_Y$ so there exists a possibly trivial torsion line bundle $\eta$ of order $2$ such that $K_Y=3H|_Y+\eta$. In particular, $K_Y$ is numerically equivalent to $3H|_Y$ which is ample since $3H$ is ample on $\bP^5$: $K_Y$ is ample too by Moishezon-Nakai criterion. Hence, $Y$ is a minimal surface of general type.
\end{proof}

We stress that $K_Y$ is not linearly equivalent to $3H|_Y$ as we will show later. Now, let us compute the main invariants of the surface $Y$:

\begin{theorem}
\label{THM:invariants}
Let $Y$ be as above. Then
\begin{enumerate}[(a)]
    \item the (topological) Euler characteristic of $Y$ is $e(Y)=357$;
    \item the Hilbert polynomial of $Y$ is $\chi(\cO_Y(n))=\frac{35}{2}n^2-\frac{105}{2}n+56$;
    \item there exists a non-trivial unramified double covering $\tilde{Y}\to Y$;
    \item $Y$ is regular (i.e. $q(Y)=0$), $p_g(Y)=55$ and $h^{1,1}(Y)=245$;
    \item $h^0(I_{Y/\bP^5}(1))=h^0(I_{Y/\bP^5}(2))=0$.
\end{enumerate}
\end{theorem}

\begin{proof}
Let us start with the computation of $e(Y)$. One can compute $e(Y)$ by using a computer algebra software (we will use this approach later in order to compute some cohomology groups), but actually here we adapt a formula presented and proved in \cite[Proposition 7.13]{Pra}, which in our specific case is
$$c_2(Y)=\sum_{(i_1,i_2)}(-1)^{i_1+i_2}((i_1+1,i_2))Q_{(i_1+2,i_2+1)}\left(\cO_{\bP^5}^6\left(1/2\right)\right)c_{2-i_1-i_2}(\bP^5)$$
where $(i_1,i_2)$ range in $\{(2,0),(1,1),(1,0),(0,0)\}$ and where $\cO_{\bP^5}(1/2)$ is a formal line bundle whose square is $\cO_{\bP^5}(1)$ (one can be more precise by invoking squaring principle). For brevity we don't report the definition of the coefficients $((a,b))$ and of the $Q$-Schur polynomial $Q_{(a,b)}$. Nevertheless, we give the values of the non-vanishing $Q$-Schur polynomial evaluated in $\cO_{\bP^5}^6(1/2)$ and of the coefficients $((a,b))$ appearing in the above formula:
$$
\begin{array}{cccc}
Q_{2,1}=35H^3 &
Q_{3,1}=105H^4 & 
Q_{4,1}=\frac{777}{4}H^5 & 
Q_{3,2}=\frac{483}{4}H^5 \\
((1,0))=1 &
((2,0))=3 &
((3,0))=7 &
((2,1))=3 
\end{array}
$$
Summing up and developing the computation, one gets $e(Y)=\deg(c_2(Y))=357$.
\smallskip

By Lemma \ref{LEM:BasicFactsOnY}(c) we have that $Y$, as subvariety of $\bP^5$, has degree $35$. Hence we have $H|_Y^2=35$. As $K_Y\equiv_{num}3H|_Y$ by Lemma \ref{LEM:BasicFactsOnY}(d), we have that $K_Y^2=315$ and $K_Y\cdot H|_Y=105$. Then we have
$$\chi(\cO_Y)=\frac{e(Y)+K_Y^2}{12}=\frac{357+315}{12}=56$$
by Noether's formula and 
$$\chi(\cO_Y(n))=56+\frac{1}{2}(nH|_Y)(nH|_Y-K_Y)=
56+\frac{1}{2}((H|_Y^2)n^2-(H|_Y\cdot K_Y)n)=\frac{35}{2}n^2-\frac{105}{2}n+56.$$
We can apply the double cover construction of Subsection \ref{SUBSEC:DoubleCovers} to $Y=\cD_4(f)=\cD_4'(\varphi_{H_f})$ in order to construct the threefold $Z$, the unramified double covering $\beta:\tilde{Y}\to Y$ and the morphisms $\pi$ and $\alpha$ (see Diagram \eqref{EQ:DiagDoubleCover}). Recall that, by Proposition \ref{PROP:RegularityAndUnramifiedCovering} in order to see that $h^1(\cO_Z)=0$ and that $\beta:\tilde{Y}\to Y$ is indeed a non-trivial unramified double covering. 
Since $\alpha_*\cO_{Z}=\cO_{\tilde{Y}}$ by construction and as $H^1(\cO_Z)=0$, by the Leray spectral sequence one has $q(\tilde{Y})=h^1(\cO_{\tilde{Y}})=0$.
From the surjectivity of $\beta:\tilde{Y}\to Y$, we have that $\beta^*:H^0(\Omega_Y^1)\to H^0(\Omega_{\tilde{Y}}^1)$ is injective so $q(Y)=0$. 
The last two invariants, namely $h^{1,1}(Y)$ and $p_g(Y)$, are easily computed knowing that $\chi(\cO_Y)=56$, $q(Y)=0$ and $e(Y)=357$. 

Claim $(e)$ follows since a non-trivial section of $\cI_{Y/\bP^5}(d)$ induces, via pullback, a non-trivial section of $\cI_{Z/\bG}\otimes \pi_2^*\cO_{\bP^5}(d)$ and we know by Proposition \ref{PROP:RegularityAndUnramifiedCovering} that $h^0(\cI_{Z/\bG}\otimes\pi_2^*\cO_{\bP^5}(d))=0$ for $d=0,1$.
\end{proof}
 
The method described in Subsection \ref{SUBSEC:Application} is not powerful enough to prove the vanishing $h^0(\cI_{Z/\bG}\otimes\pi_2^*\cO_{\bP^5}(d))=0$ for $d=3,4$ and so we cannot use this method in order to conclude that $h^0(\cI_{Y/\bP^5}(d))=0$ for $d=3,4$. Nevertheless, by semicontinuity it is enough to establish the vanishing for a single example in order to have it for the general one. With this approach, one can use a computer algebra software (like {\it Magma}) in order to compute the Hilbert series $h_{Y_0}(t)$ associated to the surface $Y_0$ for a specific case, namely the Klein cubic fourfold
$$X_0:\quad f_0=x_0^2x_1+x_1^2x_2+x_2^2x_3+x_3^2x_4+x_4^2x_5+x_5^2x_0=0.$$ 
Defining, for brevity, $x_k=x_{i}$ for any $k\in \bZ$, $i\in\{0,\dots,5\}$ if and only if $k\equiv i \mod 6$, one has
$$\hess(f_0)=\sum_{i=0}^2 x_i^3x_{i+3}^3-x_0x_1x_2x_3x_4x_5+
\sum_{i=0}^5 x_{i}x_{i+1}^3x_{i+3}^2
-\sum_{i=0}^5 x_{i}x_{i+1}x_{i+2}x_{i+3}^3
-\sum_{i=0}^1 x_{i}^2x_{i+2}^2x_{i+4}^2.
$$
Let us observe that $Y_0=\cD_{4}(f_0)$ has the expected dimension: it can be proved directly by using the same reasoning presented in \cite{AR} for the case of the Klein threefold. Using {\it Magma} one obtain
\begin{equation}
\label{EQ:HILB}
h_{Y_0}(t)=\frac{15t^4+10t^3+6t^2+3t+1}{(1-t)^3}=\sum_{i=0}^{\infty} h_{Y_0}^{(i)}t^i=1+6t+21t^2+56t^3+126t^4+231t^5\, \mod t^6.
\end{equation}

Actually, this is also the Hilbert series for $Y$ general, since $h_Y(t)$ is constant for flat families (and we are considering degeneracy loci associated to a morphism of vector bundles). This has several consequences.

\begin{proposition}
\label{PROP:AbbiamoCedutoAlLatoOscuro}
Let $[f]\in U$ be general and let $Y$ be as above. Then
\begin{enumerate}[(a)]
    \item The $2$-torsion element $\eta$ such that $K_Y=3H|_Y+\eta$ is non-trivial;
    \item $h^0(\cI_{Y/\bP^5}(d))=0$ for $d\leq 4$ and $h^0(\cI_{Y/\bP^5}(5))=21$.
\end{enumerate}
\end{proposition}

\begin{proof}
In order to prove $(a)$, notice that the coefficient $h_Y^{(3)}$ of the Hilbert series of $Y$ is $56$, which equals, by definition, the dimension of $S_Y^{(3)}$, the degree $3$ part of the homogeneous coordinate ring $S_Y$ of $Y$, i.e. $S/I_{Y/\bP^5}$. On the other hand, $S_Y^{(3)}$ is the image of 
$S^3=H^0(\cO_{\bP^5}(3))$ in $H^0(\cO_{Y}(3))$ via the map induced by the exact sequence
$$0\to \cI_{Y/\bP^5}(3)\to \cO_{\bP^5}(3)\to \cO_Y(3)\to 0.$$
Then, if $K_Y=3H|_Y$ we would have a contradiction since we would obtain, by Theorem \ref{THM:invariants}, $55=p_g(Y)=h^0(\cO_Y(3))\geq 56$. Hence $\eta$ is a non-trivial $2$-torsion element of $\Pic(Y)$.
\smallskip

For $(b)$, notice that $h_Y^{(d)}=\dim(S^d)=h^0(\cO_{\bP^5}(d))$ for $d\leq 4$. Hence $H^0(\cI_{Y/\bP^5}(d))$, which equals the kernel of the map $S^d\to H^0(\cO_Y(d))$, is trivial for $d\leq 4$. For $d=5$ one has $h_Y^{(5)}=231=252-21=\dim(S^5)-21$ so $h^0(\cI_{Y/\bP^5}(5))=21$ with the same argument as before.
\smallskip
\end{proof}

Recall that we proved in Subsection \ref{SUBSEC:DoubleCovers} that $Y$ has a natural non-trivial unramified double cover. This corresponds to a $2$-torsion line bundles $\eta'$ on $Y$.
An intriguing question is whether $\eta$ and $\eta'$ coincide. We conjecture the following:

\begin{conj}
We have $\eta=\eta'$ for $[f]\in U$ general.
\end{conj}

We conclude this section by exploiting again {\it Magma} in order to obtain the following data which hold for $Y$ associated to $[f]\in U$ general:

\begin{equation}
\label{EQ:CohTable}
\begin{array}{c|cccc}
d     & 0 & 1 & 2 & 3 \\ \hline
h^0(\cO_Y(d))     & 1 & 6 & 21 & 56\\
h^1(\cO_Y(d))     & 0 & 0 & 0 & 0\\
h^2(\cO_Y(d))     & 55 & 15 & 0 & 0
\end{array}
\end{equation}

Using this, one has:

\begin{proposition}
\label{PROP:AbbiamoCedutoAlLatoOscuroLaVendetta}
Let $[f]\in U$ be general and let $Y$ be as above. Then $Y$ is projectively normal.
\end{proposition}

\begin{proof}
Recall that if $S_Y=S/I_{Y/\bP^5}$ is the homogeneous coordinate ring of $Y$, we have, for each $d\geq 0$, an exact sequence of vector spaces
$$0\to S_Y^{(d)}\to H^0(\cO_Y(d))\to H^1(\cI_{Y/\bP^5}(d))\to 0.$$
From these sequences one has
$$\sum_{d=0}^{+\infty}h^0(\cO_Y(d))t^d=h_Y(t)+\sum_{d=0}^{+\infty}h^1(\cI_{Y/\bP^5}(d))t^d.$$

Since $dH|_Y\equiv_{num} K_Y+(d-3)H|_Y$ and $(d-3)H|_Y$ is ample for $d\geq 4$, by Kodaira vanishing, one has $H^p(\cO_{Y}(d))=0$ for $d\geq 4$ and $p=1,2$. In particular, using also Table \eqref{EQ:CohTable}, one has
$$\sum_{d=0}^{+\infty}h^0(\cO_Y(d))t^d=\sum_{d=0}^{+\infty}\chi(\cO_Y(d))t^d-(55+15t)=\frac{7(18t^2-21t+8)}{(1-t)^3}-55-15t.$$
One can easily check that the latter series coincides with $h_Y(t)$ (see Equation \eqref{EQ:HILB}) so one can conclude that $h^1(\cI_{Y/\bP^5}(d))=0$ for all $d\geq 0$. This is equivalent to the projective normality of $Y$.
\end{proof}

\end{document}